\documentclass[12pt,twoside,reqno]{amsart}

\usepackage[OT1]{fontenc}    %
\usepackage{type1cm}         %
\usepackage{amsthm}          %
\usepackage[dvips]{graphicx} 
\usepackage[bf]{caption}     
\usepackage{psfrag}          

\usepackage{amsfonts}
\usepackage{amsmath}
\usepackage{amssymb}
\usepackage{wasysym}

\usepackage{verbatim}        

\numberwithin{equation}{section}

\theoremstyle{plain}
\newtheorem{theorem}{Theorem}[section]
\newtheorem{corollary}[theorem]{Corollary}
\newtheorem{proposition}[theorem]{Proposition}
\newtheorem{lemma}[theorem]{Lemma}

\theoremstyle{remark}
\newtheorem{remark}[theorem]{Remark}

\newtheorem{example}[theorem]{Example}
\newtheorem*{ack}{Acknowledgement}

\theoremstyle{definition}
\newtheorem{definition}[theorem]{Definition}

\newcommand{\PA}{\mathcal{A}}
\newcommand{\PB}{\mathcal{B}}
\newcommand{\HH}{\mathcal{H}}
\newcommand{\LL}{\mathcal{L}}
\newcommand{\QQ}{\mathcal{Q}}
\newcommand{\R}{\mathbb{R}}

\newcommand{\N}{\mathbb{N}}

\newcommand{\iii}{\mathtt{i}}
\newcommand{\jjj}{\mathtt{j}}
\newcommand{\kkk}{\mathtt{k}}
\newcommand{\eps}{\varepsilon}
\newcommand{\fii}{\varphi}
\newcommand{\roo}{\varrho}
\newcommand{\ualpha}{\overline{\alpha}}
\newcommand{\lalpha}{\underline{\alpha}}

\newcommand{\id}{\textrm{Id}}
\newcommand{\intI}{\mathcal{I}}
\newcommand{\intJ}{\mathcal{J}}

\DeclareMathOperator{\dimm}{dim_M}
\DeclareMathOperator{\dimum}{\overline{dim}_M}
\DeclareMathOperator{\dimlm}{\underline{dim}_M}
\DeclareMathOperator{\dimh}{dim_H}

\DeclareMathOperator{\diam}{diam}

\DeclareMathOperator{\conv}{conv}

\begin{document}

\title{Overlapping self-affine sets of Kakeya type}

\author{ Antti K\"aenm\"aki      \and
         Pablo Shmerkin               }

\address{Department of Mathematics and Statistics \\
         P.O. Box 35 (MaD) \\
         FI-40014 University of Jyv\"askyl\"a \\
         Finland}

\email{antakae@maths.jyu.fi}
\email{shmerkin@maths.jyu.fi}

\thanks{AK acknowledges the support of the Academy of Finland (project
  \#114821). Research of PS was partially supported by NSF
  grant \#DMS-0355187 and the Academy of
  Finland. Part of this research was carried out while PS was visiting the Instituto de Matem\'{a}tica Pura e Aplicada (IMPA), Brazil.}
\date{\today}

\begin{abstract}
  We compute the Minkowski dimension for a family of self-affine sets
  on $\R^2$. Our result holds for every (rather than generic) set in
  the class. Moreover, we exhibit explicit open subsets of this
  class where we allow overlapping, and do not impose any
  conditions on the norms of the linear maps. The family under consideration was inspired by the
  theory of Kakeya sets.
\end{abstract}

\maketitle

\section{Introduction}

An \emph{iterated function system (IFS)} on $\R^d$ is a finite
collection of strictly contractive self-maps $f_1,\ldots,f_\kappa$. A
classical result, formalized by Hutchinson \cite{Hutchinson1981}
(although the crucial idea goes back to Moran \cite{Moran1946}),
states that for every IFS there is a unique nonempty compact set $E
\subset \R^d$ for which
\begin{equation*}
  E = \bigcup_{i=1}^\kappa f_i(E).
\end{equation*}
When the mappings are similitudes (or conformal) and the pieces
$f_i(E)$ do not overlap much, the Hausdorff dimension of $E$ is
easily determined by the contraction ratios of the mappings $f_i$,
see for example \cite{Hutchinson1981}, \cite{MauldinUrbanski1996},
and \cite{KaenmakiVilppolainen2006}. In the present article, we
assume that the mappings $f_i$ are affine; in this case the set
$E$ is called a \emph{self-affine set}. In addition, we do not
require any non-overlapping condition. Dropping either the
conformality or separation hypothesis makes the problem of
estimating dimension dramatically more complicated. The main
feature of our work is that we are able to drop both, while
obtaining results which are valid everywhere, not just
generically.

The so-called singular value function plays a prominent r\^ole in
the study of the dimension of self-affine sets. Following
\cite[Proposition 4.1]{Falconer1988}, the singular value function
leads to a notion of the singular value dimension, which serves as
an upper bound for the upper Minkowski dimension, see
\cite{DouadyOesterle1980} and \cite{Falconer1988}.
Falconer \cite{Falconer1988} (see also \cite{Solomyak1998}) proved
that assuming the norms of the linear parts to be less than
$\tfrac12$, this upper bound is sharp, and also equals the
Hausdorff dimension, for $\LL^{d\kappa}$-almost every choice of
translation vectors.
Here $\LL^{d\kappa}$ denotes the Lebesgue measure on
$\R^{d\kappa}$. Falconer and Miao \cite{FalconerMiao2007a} have
recently shown that the size of the set of exceptional translation
vectors is small also in the sense of Hausdorff dimension. The
self-affine carpets of McMullen \cite{McMullen1984} show that one
cannot replace ``almost all'' by ``all'', even if the pieces do
not overlap. Furthermore, it follows from examples in
\cite{Edgar1992} that the $\tfrac12$ bound on the norms is
essential. These counterexamples are of a very special kind, and
it is therefore of interest to find families of self-affine sets
for which one can loose these assumptions.

A result into this direction was obtained by Hueter and Lalley in
\cite{HueterLalley1995}, where it is proven that for an explicit
open class of self-affine sets, the Hausdorff dimension is indeed
given by the singular value dimension, as long as the pieces
$f_i(E)$ are disjoint. In their result the norms may be greater
than $\tfrac12$, but it follows from their hypotheses that the
singular value dimension is less than $1$. In a different
direction, it was recently proven in
\cite{JordanPollicottSimon2006} that for a randomized version of
self-affine sets the natural analogue of Falconer's formula holds
almost surely regardless of the norms. See also \cite{Hu1998},
\cite{Kaenmaki2004}, \cite{FengWang2005}, \cite{Shmerkin2005}, and
\cite{FalconerMiao2007b} for other recent results on the
dimensional properties of self-affine sets.

For a fixed $\kappa$ and $d$, the class of all IFSs consisting of
$\kappa$ affine maps on $\R^d$ inherits a natural topology from
$\mathcal{A}_d^\kappa$, where $\mathcal{A}_d=GL_d(\R)\times \R^d$
is identified with the vector space of all invertible affine
mappings on $\R^d$. We will say that a family of affine IFS's is
\emph{robust} if it is open in this topology, and that a property
is \emph{stable} if the set of IFS's where it holds is robust.

We define a class of self-affine sets in which we allow
overlapping and the norms of all the maps can be arbitrarily close
to $1$; see \S \ref{sec:kakeya} for the details. We show that in
this class the Minkowski dimension coincides with the singular
value dimension (Theorem \ref{thm:main_result}), and it can be
defined dynamically as the zero of a certain pressure function.
Even though the family is not itself robust, in \S
\ref{sec:examples} we will exhibit robust subsets which preserve
all the interesting properties. This is the first instance where
the equality of Minkowski dimension and singular value dimension
is established for a robust family, without requiring any
separation assumptions. Moreover, we prove that the Minkowski
dimension is a continuous function of the generating maps within
this family.

The inspiration for our work arose from the theory of Kakeya sets.
Recall that a subset of $\R^d$ is called a \emph{Kakeya set}
(sometimes also a \emph{Besicovitch set}) if it contains a unit
segment in every direction. The long-standing Kakeya conjecture
asserts, in one of its many forms, that the Hausdorff dimension of
a Kakeya set in $\R^d$ is precisely $d$. This is wide open for
$d\ge 3$; however, for $d=2$ it is known to be true, and indeed
the proof is not difficult, see for example \cite{Wolff1999}. This
result implies that the overlap between segments pointing in
different directions is small, in the sense that the dimension of
the union of all segments is the same as if there was no overlap
at all. We strove to construct a family of self-affine sets in
which the cylinder sets are aligned in different directions, so
that the possible overlaps between them would not affect the
dimension calculations. Although the technical details may obscure
it somewhat, it may be useful to keep this basic idea in mind
while going through the definitions and proofs.

The paper is structured as follows. In \S \ref{sec:self-affine}, we
introduce some standard notation and present some preliminary
facts on self-affine sets. The family of self-affine sets of
Kakeya type is defined in \S \ref{sec:kakeya}, where Theorem
\ref{thm:main_result}, the main result of the paper, is stated.
The proof of Theorem \ref{thm:main_result} is contained in \S
\ref{sec:proof}. In \S \ref{sec:projection}, we study projections
of self-affine sets, as part of our preparation to obtain explicit
examples of self-affine sets of Kakeya-type. These examples are
introduced in \S \ref{sec:examples}, where we finish our
discussion with some remarks and open questions.

\section{Self-affine sets} \label{sec:self-affine}

Throughout the article, we use the following notation: Let $0 <
\ualpha < 1$ and $I = \{ 1,\ldots,\kappa \}$ with $\kappa \ge 2$. Put
$I^* = \bigcup_{n=1}^\infty I^n$ and $I^\infty =
I^\N$. For each $\iii \in I^*$, there is $n \in \N$ such that
$\iii = (i_1,\ldots,i_n) \in I^n$. We call this $n$ as the
\emph{length} of $\iii$ and we denote $|\iii|=n$. The length of
elements in $I^\infty$ is infinity. Moreover, if $\iii \in I^*$
and $\jjj \in I^* \cup I^\infty$ then with the notation $\iii\jjj$
we mean the element obtained by juxtaposing the terms of $\iii$
and $\jjj$. For $\iii \in I^*$, we define $[\iii] = \{ \iii\jjj :
\jjj \in I^\infty \}$ and we call the set $[\iii]$ a
\emph{cylinder
  set} of level $|\iii|$. If $\jjj \in I^* \cup I^\infty$ and $1 \le n
< |\jjj|$, we define $\jjj|_n$ to be the unique element $\iii \in
I^n$ for which $\jjj \in [\iii]$. We also denote $\iii^- =
\iii|_{|\iii|-1}$. With the notation $\iii \bot \jjj$, we mean
that the elements $\iii,\jjj \in I^*$ are \emph{incomparable},
that is, $[\iii] \cap [\jjj] = \emptyset$. We call a set $A
\subset I^*$ incomparable if all of its elements are mutually
incomparable. Finally, with the notation $\iii \land \jjj$, we
mean the common beginning of $\iii \in I^*$ and $\jjj \in I^*$,
that is, $\iii \land \jjj = \iii|_n = \jjj|_n$, where $n = \min\{
k-1 : \iii|_k \ne \jjj|_k \}$.

Defining
\begin{equation*}
  |\iii - \jjj| =
  \begin{cases}
    \ualpha^{|\iii \land \jjj|}, \quad &\iii \ne \jjj \\
    0, &\iii = \jjj
  \end{cases}
\end{equation*}
whenever $\iii,\jjj \in I^\infty$, the couple $(I^\infty,|\cdot|)$
is a compact metric space. We call $(I^\infty,|\cdot|)$ a
\emph{symbol
  space} and an element $\iii = (i_1,i_2,\ldots) \in I^\infty$ a
\emph{symbol}. If there is no danger of misunderstanding, we will
also call an element $\iii \in I^*$ a symbol. Define the
\emph{left shift} $\sigma \colon I^\infty \to I^\infty$ by setting
\begin{equation}
 \sigma(i_1,i_2,\ldots) = (i_2,i_3,\ldots).
\end{equation}
The notation $\sigma(i_1,\ldots,i_n)$ means the symbol
$(i_2,\ldots,i_n) \in I^{n-1}$. Observe that to be precise in our
definitions, we need to work with ``empty symbols'', that is,
symbols with zero length, which will be denoted by $\varnothing$.

The singular values $1 > ||A|| = \alpha_1(A) \ge \cdots \ge \alpha_d(A) > 0$
of a contractive invertible matrix $A \in \R^{d\times d}$ are the
square roots of the eigenvalues of the positive definite matrix
$A^*A$, where $A^*$ is the transpose of $A$. The normalized
eigenvectors of $A^*A$ are denoted by
$\theta_1(A),\ldots,\theta_d(A)$. These eigenvectors together with
singular values give geometric information about the matrix $A$. For
example, let $v$ be the unit
vector with direction equal to the major axis of the ellipse $A(B)$,
where $B$ is any ball. By definition, the direction of $v$ is the
image under $A$ of a vector which maximizes $|A x|$ over all $x$ in
the unit ball. But $\theta_1(A)$ is precisely such a vector since $|A
x|^2 = A^* A x \cdot x$. Thus, explicitly, $v =
A\bigl( \theta_1(A) \bigr)/\alpha_1(A)$. For more detailed
information, the reader is referred to \cite[\S V.1.3]{Temam1988}.

For a contractive invertible matrix $A \in \R^{d\times d}$,
we define the \emph{singular value function} to be
\begin{equation*}
  \fii^t(A) = \alpha_1(A) \cdots \alpha_l(A) \alpha_{l+1}(A)^{t-l},
\end{equation*}
where $0 \le t < d$ and $l$
is the integer part of $t$. For $t \ge d$, we put $\fii^t(A) =
\bigl(\alpha_1(A) \cdots \alpha_d(A)\bigr)^{t/d} = |\det(A)|^{t/d}$.

For each $i \in I$, fix a contractive invertible matrix $A_i \in
\R^{d\times d}$ such that $||A_i|| \le \ualpha < 1$. Clearly
the products $A_\iii = A_{i_1}\cdots A_{i_n}$ are also
contractive and invertible as $\iii \in I^n$ and $n \in
\N$. Denoting 
$\lalpha = \min_{i \in I} \alpha_d(A_i) > 0$, for each $t,\delta
\ge 0$ we have
\begin{equation} \label{eq:cylinder1}
  \fii^t(A_\iii)\lalpha^{\delta |\iii|} \le \fii^{t+\delta}(A_\iii)
  \le \fii^t(A_\iii)\ualpha^{\delta |\iii|}
\end{equation}
whenever $\iii \in I^*$. According to \cite[Corollary
V.1.1]{Temam1988} and \cite[Lemma 2.1]{Falconer1988}, the
following holds for all $t \ge 0$:
\begin{equation} \label{eq:cylinder2}
  \fii^t(A_{\iii\jjj}) \le \fii^t(A_\iii) \fii^t(A_\jjj)
\end{equation}
whenever $\iii,\jjj \in I^*$.

Given $t \ge 0$, we define the \emph{topological pressure} to be
\begin{equation} \label{eq:pressure}
  P(t) = \lim_{n \to \infty} \tfrac{1}{n} \log\sum_{\iii \in I^n}
  \fii^t(A_\iii).
\end{equation}
The limit above exists by the standard theory of subadditive sequences
since for each $t \ge 0$, using \eqref{eq:cylinder2},
\begin{equation*}
  \sum_{\iii \in I^{n+m}} \fii^t(A_\iii) \le \sum_{\iii \in I^{n+m}}
  \fii^t(A_{\iii|_n})\fii^t(A_{\sigma^n(\iii)}) = \sum_{\iii \in I^n}\fii^t(A_\iii)
  \sum_{\jjj \in I^m}\fii^t(A_\jjj)
\end{equation*}
whenever $n,m \in \N$.
Moreover, as a function, $P \colon [0,\infty) \to \R$ is continuous
and strictly decreasing with $\lim_{t \to \infty}P(t) = -\infty$:
For $t,\delta \ge 0$ and $n \in \N$, we have, using \eqref{eq:cylinder1},
\begin{equation*}
  \delta\log\lalpha + \tfrac{1}{n} \log\sum_{\iii \in I^n} \fii^t(A_\iii)
  \le \tfrac{1}{n} \log\sum_{\iii \in I^n} \fii^{t+\delta}(A_\iii)
  \le \delta\log\ualpha + \tfrac{1}{n} \log\sum_{\iii \in I^n}
  \fii^t(A_\iii).
\end{equation*}
Letting $n \to \infty$, we get $0 < -\delta\log\ualpha \le P(t) -
P(t+\delta) \le -\delta\log\lalpha$. Since $P(0) = \log\kappa$, we
have actually shown that there exists a unique $t > 0$ for which $P(t)=0$.
The \emph{singular value dimension} is defined to be the zero of the
topological pressure. See also \cite[Proposition 4.1]{Falconer1988}.

\begin{theorem} \label{thm:semiconformal}
  Suppose that for each $i \in I$ there is an invertible
  matrix $A_i \in \R^{d\times d}$ with $||A_i|| \le \ualpha$. If for
  given $t \ge 0$ there exists a constant $D \ge 1$ such that
  \begin{equation*}
    D^{-1}\fii^t(A_\iii)\fii^t(A_\jjj) \le \fii^t(A_{\iii\jjj})
  \end{equation*}
  whenever $\iii,\jjj \in I^*$ then there exists a
  Borel probability measure $\mu$ on $I^\infty$, a constant $c \ge
  1$, and $1 > \lambda_1(\mu) \ge \cdots \ge \lambda_d(\mu) > 0$
  such that
  \begin{equation}
    \label{eq:semiconformal}
    c^{-1} e^{-|\iii|P(t)}\fii^t(A_\iii) \le \mu([\iii]) \le
    ce^{-|\iii|P(t)}\fii^t(A_\iii)
  \end{equation}
  whenever $\iii \in I^*$ and
  \begin{equation*}
    \lim_{n \to \infty} \alpha_k(A_{\iii|_n})^{1/n} = \lambda_k(\mu)
  \end{equation*}
  for $\mu$-almost all $\iii \in I^\infty$ and for every $k \in
  \{ 1,\ldots,d \}$.
\end{theorem}

\begin{proof}
  Using the assumptions, \eqref{eq:cylinder1},
  and \eqref{eq:cylinder2}, the existence of a 
  Borel probability measure $\mu$ satisfying \eqref{eq:semiconformal}
  follows from \cite[Theorem 2.2]{KaenmakiVilppolainen2006} by a minor
  modification. More precisely, in \cite{KaenmakiVilppolainen2006} it
  was assumed that the parameter $t$ is an exponent, but an
  examination of the proof reveals that this fact is not
  required. Using \cite[Theorem 2.2]{KaenmakiVilppolainen2006},
  \eqref{eq:cylinder2},
  and Kingman's subadditive ergodic theorem \cite{Steele1989},
  the limit
  \begin{equation*}
    E^t(\mu) = \lim_{n \to \infty} \tfrac1n \log\fii^t(A_{\iii|_n})
  \end{equation*}
  exists for $\mu$-almost every $\iii \in I^*$ and for every $t \ge 0$.
  Setting now $\lambda_k(\mu) = \exp\bigl( E^k(\mu)-E^{k-1}(\mu) \bigr)$
  for $k \in \{ 1,\ldots,d \}$, we have finished the proof.
\end{proof}

It may appear that the assumption of Theorem
\ref{thm:semiconformal} is very strong. However, it is implied by
some simple geometrical conditions; see Remark
\ref{P-semiconformal}. Observe also that even if the measure
satisfying \eqref{eq:semiconformal} did not exist, the latter
claim of Theorem \ref{thm:semiconformal} remains true for the
natural measure found in \cite[Theorem 4.1]{Kaenmaki2004}.

If for each $i \in I$ an invertible matrix $A_i \in \R^{d \times d}$
with $||A_i|| \le \ualpha$ and a translation vector $a_i$ are fixed
then we define a \emph{projection mapping} $\pi \colon I^\infty \to \R^d$
by setting
\begin{equation*}
  \pi(\iii) = \sum_{n=1}^\infty A_{\iii|_{n-1}}a_{i_n}
\end{equation*}
as $\iii = (i_1,i_2,\ldots)$. Using the triangle inequality, we have
\begin{align*}
  |\pi(\iii) - \pi(\jjj)| &\le \sum_{n=|\iii \land \jjj|+1}^\infty
  |A_{\iii|_{n-1}}a_{i_n} - A_{\jjj|_{n-1}}a_{j_n}| \\
  &\le \sum_{n=|\iii \land \jjj|+1}^\infty
  2\ualpha^{n-1} \max_{i \in I} |a_i|
  = \frac{2\max_{i \in I}|a_i|}{1-\ualpha}|\iii - \jjj|
\end{align*}
for every $\iii,\jjj \in I^\infty$. The mapping $\pi$ is therefore
continuous.

We define $E = \pi(I^\infty)$ and call this set a
\emph{self-affine set}. Observe that the compact set $E$ is
invariant under the affine mappings $A_i + a_i$, that is,
\begin{equation} \label{eq:invariance}
  E = \bigcup_{i=1}^\kappa (A_i+a_i)(E).
\end{equation}
This is an immediate consequence of the fact that
\begin{equation*}
  \pi(i\iii) = (A_i + a_i)\sum_{n=1}^\infty A_{\iii|_{n-1}}a_{i_n}
  = (A_i + a_i)\pi(\iii)
\end{equation*}
whenever $\iii \in I^\infty$ and $i \in I$. In fact, by \cite[\S
3.1]{Hutchinson1981}, there are no other nonempty compact sets
satisfying \eqref{eq:invariance} besides $E$. If there is no
danger of misunderstanding, the image of a cylinder set
\[
\pi([\iii]) = (A_{i_1} + a_{i_1}) \cdots (A_{i_n} + a_{i_n})(E) =
A_\iii(E) + A_{\iii|_{n-1}}a_{i_n} + \cdots + a_{i_1},
\]
as $\iii = (i_1,\ldots,i_n) \in I^n$, will also be called a
cylinder set, and we will denote $E_\iii = \pi([\iii])$. When we
want to emphasize the dependence of $E$ on the affine mappings, we
will say that $E$ is the \emph{invariant set} of the affine IFS
$\{ A_i + a_i\}_{i\in I}$.

\section{Self-affine sets of Kakeya type} \label{sec:kakeya}

In this section, we introduce self-affine sets of Kakeya type.
Working in $\R^2$, we state that the Minkowski dimension of such a
set is the zero of the topological pressure, see
\eqref{eq:pressure}. Given a set $A \subset \R^d$, the upper and
lower Minkowski dimensions are denoted by $\dimum(A)$ and
$\dimlm(A)$, respectively. For the definition, see \cite[\S
5.3]{Mattila1995}. If $\dimum(A) = \dimlm(A)$, then the common
value, the Minkowski dimension, is denoted by $\dimm(A)$. For $\theta
\in S^{d-1} = \{ x \in \R^d : |x| = 1 \}$ and $0 \le \beta \le \pi$,
we set
\begin{equation*} 
  X(\theta,\beta) = \{ x \in \R^d : \cos(\beta/2) < |\theta \cdot x| /
  |x|,\; x \ne 0 \}.
\end{equation*}
The closure of a given set $A$ is denoted by $\overline{A}$ and with
the notation $\LL^d$, we mean the Lebesgue measure on $\R^d$.

\begin{definition} \label{def:kakeya}
  Suppose that for each $i \in I$ there are a contractive invertible
  matrix $A_i \in \R^{2 \times 2}$ with $||A_i|| \le \ualpha<1$ and a
  translation vector $a_i \in \R^2$. The collection of affine mappings
  $\{ A_i + a_i \}_{i \in I}$ is called an \emph{affine iterated
    function system of Kakeya type}, and the invariant set $E \subset
  \R^2$ of this affine IFS a \emph{self-affine set of Kakeya type},
  provided that the following two conditions hold:
  \begin{enumerate}
  \renewcommand{\labelenumi}{(K\arabic{enumi})}
  \renewcommand{\theenumi}{K\arabic{enumi}}
    \item \label{K-kakeya}
      There exist $\theta \in S^1$ and $0 < \beta < \pi/2$ such
      that
      \begin{subequations}
        \renewcommand{\theequation}{\ref{K-kakeya}\alph{equation}}
        \begin{align}
          A_i\bigl( \overline{X(\theta,\beta)} \bigr) &\subset
          X(\theta,\beta),  \label{K-coneinvariance} \\
          A_i^*\bigl( \overline{X(\theta,\beta)} \bigr) &\subset
          X(\theta,\beta) \label{K-transpose}
        \end{align}
        whenever $i \in I$ and
        \begin{equation} \label{K-separation}
          A_i\bigl( \overline{X(\theta,\beta)}\bigr) \cap A_j\bigl(
          \overline{X(\theta,\beta)} \bigr) = \{0\}
        \end{equation}
        for $i \ne j$.
      \end{subequations}
    \item \label{K-projection}
      There exists a constant $\roo>0$ such that
      \begin{equation*}
        \LL^1 \bigl( \{ \theta_1(A_\iii) \cdot x : x \in E \} \bigr)
        \ge \roo
      \end{equation*}
      for all $\iii \in I^*$.
  \end{enumerate}
\end{definition}

Let us make some remarks on these conditions. Our goal is to make
the self-affine set look, at a given finite scale, roughly like a
rescaled Kakeya set (except that instead of having segments in
every direction, there are segments only in a Cantor set of
directions). The r\^{o}le of the conditions
\eqref{K-coneinvariance} and \eqref{K-separation} is to ensure
that cylinder sets are aligned in different directions. Notice the
analogy between these conditions and the Hypothesis 3
(``separation'') in \cite{HueterLalley1995}. The hypothesis
\eqref{K-transpose} is of technical nature. We underline that
\eqref{K-coneinvariance}, \eqref{K-transpose}, and
\eqref{K-separation} are all stable properties.

The projection condition \eqref{K-projection} is needed so that
cylinder sets do not have too many ``holes'' and one can
approximate them by neighborhoods of segments. It is the only one
of the assumptions which involves the translation vectors
$\{a_i\}_{i\in I}$ in addition to the linear maps $\{A_i\}_{i\in
I}$. In particular, \eqref{K-projection} implies that the
Hausdorff dimension of $E$ is at least one. Hence if $t$ is such
that $P(t)=0$, then $t\ge 1$ by \cite[Proposition
5.1]{Falconer1988}.
An analogous, but stronger, projection condition was introduced by
Falconer in \cite{Falconer1992}. We remark that in that article,
unlike in our case, the open set condition is also required. The
projection condition is obviously satisfied if the invariant set
is connected. Unfortunately, determining when a self-affine set is
connected in a stable way is a very difficult problem, even when
the linear parts commute, see for example
\cite{ShmerkinSolomyak2006}. In \S \ref{sec:projection}, we introduce
easily checkable, stable conditions which imply the projection
condition.

We do not need analogues of either Hypothesis 2 (``distortion'')
or Hypothesis 5 (``strong separation'') used in
\cite{HueterLalley1995}. In that article, Hypothesis 2 plays a
crucial r\^{o}le in guaranteeing that the invariant set has
dimension less than 1. By our observation that $t\ge 1$, it cannot
possibly hold in our setting. In a sense, our examples are more
purely self-affine, since both singular values are involved in the
dimension calculations, while in \cite{HueterLalley1995} the
dimension depends only on the largest one. We stress that our
results are only for the Minkowski dimension; estimating the Hausdorff
dimension in our setting appears to be a very difficult problem.

Before stating our main result, we formulate and prove a
Kakeya-type estimate which is a crucial ingredient of the proof.
Even though it is a minor variant of \cite[Proposition
1.5]{Wolff1999}, complete details are provided for the convenience
of the reader.

\begin{proposition} \label{thm:kakeya_estimate}
  Let $R_1,\ldots,R_M \subset \R^2$ be rectangles of size
  $\alpha_1 \times \alpha_2$, with $\alpha_1 > \alpha_2$. Suppose that
  the angle between the long sides of any two rectangles is at least
  $\alpha_2/\alpha_1$. If $F \subset \R^2$ and $\tau > 0$ are such
  that $\LL^2(F \cap R_i) \ge \tau\alpha_1\alpha_2$ for every $i \in
  \{ 1,\ldots,M \}$, then
  \begin{equation*}
    \LL^2(F) \ge
    \frac{M\tau^2 \alpha_1\alpha_2}{2\sqrt{2}\pi \log(2\pi\alpha_1/\alpha_2)}.
  \end{equation*}
\end{proposition}

\begin{proof}
  Given two rectangles $R_i$ and $R_j$, let us denote the (smaller)
  angle between their long sides by $\varangle(R_i,R_j)$. Since
  $\alpha_2/\alpha_1 \le \varangle(R_i,R_j) \le \pi/2$, a simple geometric
  inspection yields $\alpha_2/\alpha_1 + \varangle(R_i,R_j) \le
  2\varangle(R_i,R_j) \le \sqrt{2}\pi\sin\bigl( \varangle(R_i,R_j)/2
  \bigr) = \sqrt{2}\pi\alpha_2/\diam(R_i \cap R_j)$ and hence
  \begin{equation*}
    \LL^2(R_i \cap R_j) \le \alpha_2\diam(R_i \cap R_j) \le
    \frac{\sqrt{2}\pi\alpha_2^2}{\alpha_2/\alpha_1 + \varangle(R_i,R_j)}
  \end{equation*}
  whenever $i \ne j$.
  Thus we have
  \begin{equation} \label{eq:L2_estimate}
    \sum_{j=1}^M \LL^2(R_i \cap R_j) \le \sum_{j=0}^{\lceil
      \frac{\pi\alpha_1}{2\alpha_2} \rceil}
    \frac{\sqrt{2}\pi\alpha_2^2}{\alpha_2/\alpha_1 +
      j\alpha_2/\alpha_1} \le
    2\sqrt{2}\pi\alpha_1\alpha_2\log(2\pi\alpha_1/\alpha_2)
  \end{equation}
  whenever $i \in \{ 1,\ldots,M \}$. Here with the notation $\lceil x
  \rceil$, we mean the smallest integer greater than $x$.
  Since, by using H\"older's inequality,
  \begin{align*}
    (M\tau\alpha_1\alpha_2)^2 &\le \biggl(\sum_{i=1}^M \LL^2(F \cap
      R_i)\biggr)^2 = \biggl(\int_{\R^2} \chi_F \sum_{i=1}^M
      \chi_{R_i} d\LL^2\biggr)^2 \\
    &\le \biggl( \int_{\R^2} \chi_F^2 d\LL^2 \biggr) \biggl(
    \int_{\R^2} \biggl( \sum_{i=1}^M \chi_{R_i} \biggr)^2 d\LL^2
    \biggr) \\ &= \LL^2(F) \sum_{i=1}^M \sum_{j=1}^M \LL^2(R_i \cap R_j), 
  \end{align*}
  the claim follows by applying \eqref{eq:L2_estimate}. Here $\chi_A$
  denotes the characteristic function of a given set $A$.
\end{proof}


We can now state the main result of this article.

\begin{theorem} \label{thm:main_result}
  Suppose $E \subset \R^2$ is a self-affine set of Kakeya type and
  $P(t)=0$. Then
  \begin{equation*}
    \dimm(E) = t \ge 1.
  \end{equation*}
  In particular, $\dimm$ is a continuous function when restricted
  to the class of affine IFS's of Kakeya-type.
\end{theorem}

Let us sketch the main idea of the proof; full details are postponed
until \S \ref{sec:proof}. In order to compute the Minkowski
dimension, we want to estimate the area of the set $E(\delta)$ for
small $\delta>0$, where $E(\delta)$ is the $\delta$-neighborhood
of $E$. In order to do this we take a small $r$ and decompose $E$
as a union of cylinders $\{ E_\iii \}$ with $\fii^t(A_\iii)\approx
r$ (where $t$ is the singularity dimension). The condition
\eqref{K-projection} implies that the projection of $E_\iii$ onto
the major axis of the ellipse $A_i B + a_i$ (where $B$ is some
large ball) has positive Lebesgue measure with a uniform lower
bound. Hence it follows that for large $K$ the
$K\alpha_2(A_\iii)$-neighborhood of $E_\iii$ intersects a
rectangle $R_\iii$, with small side comparable to
$\alpha_2(A_\iii)$ and long side comparable to $\alpha_1(A_\iii)$,
in a set of area comparable to $\alpha_1(A_\iii)\alpha_2(A_\iii)$.

At this point we would like to apply the Kakeya-type estimate of
Proposition \ref{thm:kakeya_estimate}. However, for this we need
all the rectangles to have the same sizes, while
$\alpha_1(A_\iii)$ and $\alpha_2(A_\iii)$ may take many different
values. We deal with this with the help of Theorem
\ref{thm:semiconformal}: with respect to the measure $\mu$ given
by that theorem, the values of $\alpha_1(A_\iii)$ and
$\alpha_2(A_\iii)$ are roughly constant for ``most'' sequences
$\iii$. More precisely, we will obtain that $\alpha_k(A_\iii)
\approx \bigl( \fii^t(A_\iii) \bigr)^{\gamma_k}$ for many sequences $\iii$,
where $\gamma_1+(t-1)\gamma_2=1$. Also, due to the Gibbs property
of $\mu$ expressed in \eqref{eq:semiconformal}, the number of
cylinders $[\iii]$ with $\fii^t(A_\iii)\approx r$ is comparable to
$r^{-1}$.

By \eqref{K-separation}, the angle between the long sides of two
of the rectangles $R_i$ and $R_j$ in the construction are sufficiently
separated. Hence we can apply Proposition \ref{thm:kakeya_estimate} and
conclude that the union of all such rectangles has Lebesgue
measure which is, up to a logarithmic factor, the same as if the
union was disjoint. Therefore, letting $\delta \approx r^{\gamma_2}$ we
conclude
\[
\LL^2\bigl( E(\delta) \bigr) \gtrsim r^{\gamma_1+\gamma_2-1+\eps}
\approx \delta^{2-t+\eps},
\]
where $\eps > 0$ is arbitrarily small, which gives the desired lower
estimate (the upper estimate is well known).
The latter claim of the theorem is now an immediate consequence of the
next lemma.

\begin{lemma} \label{thm:continuity}
  Suppose that for each $i \in I$ there is a contractive invertible
  matrix $A_i \in \R^{2 \times 2}$ such that the condition
  \eqref{K-coneinvariance} is satisfied. Then $(A_1,\ldots,A_\kappa)$
  is a continuity point for the singular value dimension.
\end{lemma}

\begin{proof}
  After an appropriate rotation we can assume, without loss of
  generality, that $\theta = \frac{1}{\sqrt{2}}(1,1)$ in the
  condition \eqref{K-coneinvariance}. This implies that for each $i
  \in I$, the coefficients of $A_i$ are either all strictly positive
  or all strictly negative, and this property is preserved under
  small perturbations. Since multiplying by the scalar $-1$ does not
  affect the singular values of $A_{\iii}$ for $\iii \in I^{\ast}$,
  we will assume that for each $i \in I$, the matrix $A_i$ has
  coefficients bounded below by some $\delta > 0$. Note that, since
  $A_i$ is contractive, all of its coefficients are bounded above by
  $1$.

  If $M_1, M_2 \in \R^{2 \times 2}$ and $c \in \R$, by $M_1 < M_2$ we mean
  that the inequality holds for each coefficient, and by $c < M_1$ we will
  mean that all coefficients of $M_1$ are strictly greater than
  $c$. In the same way
  we define $M_1 > M_2$ and $c > M_1$. Note that if $0 < M_1 < M_2$, then
  $\alpha_1 (M_1) < \alpha_1 (M_2)$ by the Perron-Frobenius
  Theorem. Fix $0 < \eps < \delta$, and suppose
  that for each $i \in I$ there is a matrix $B_i \in \R^{2 \times 2}$ such
  that
  \[ - \eps < A_i - B_i < \eps . \]
  Let $\eps_1 = \eps / \delta$, and note that
  \[ (1 - \eps_1) A_i < B_i < (1 + \eps_1) A_i . \]
  Iterating this, we get that if $\iii \in I^n$, then
  \[ (1 - \eps_1)^n A_{\iii} < B_{\iii} < (1 + \eps_1)^n A_{\iii}, \]
  and hence
  \begin{equation}
    \label{eq:ineqalpha} (1 - \eps_1)^n \alpha_1 (A_{\iii}) < \alpha_1
    (B_{\iii}) < (1 + \eps_1)^n \alpha_1 (A_{\iii}) .
  \end{equation}
  A straightforward calculation shows that, for $i \in I$,
  \[ | \det (A_i) | - 8 \eps < | \det (B_i) | < | \det (A_i) | + 8 \eps, \]
  whence, letting
  \[ \eps_2 = \max_{i \in I} 8 \eps | \det (A_i) |^{- 1}, \]
  we obtain
  \[ (1 - \eps_2) | \det (A_i) | < | \det (B_i) | < (1 + \eps_2) | \det
  (A_i) |. \]

  Recall the definition of the pressure function given in
  \eqref{eq:pressure}. Let $P_\PA$ and $P_\PB$ denote the pressures
  corresponding to the matrices $\{A_i \}_{i \in I}$ and $\{B_i \}_{i \in
    I}$, respectively. Let
  $t$ be such that $P_\PA(t)=0$, and let $s$ be
  such that $P_\PB(s)=0$. Our goal is to show that
  $s \to t$ as $\eps \downarrow 0$.

  Let $D=\max_{i\in I} |\det(A_i)|$. Pick any $D' \in
  (D,1)$, and suppose $\eps$ is so small that
  $D+8\eps < D'$. If $s \ge 2$, then it is easy to see that
  the pressure is given by
  \[
  P_\PB(s) = \log\biggl(\sum_{i\in I} |\det(B_i)|^{s/2}\biggr)
  \le \log\kappa - \tfrac{s}{2} |\log D'|.
  \]
  Using this, we see that
  \[
  s \le \max(2\log\kappa/|\log D'|,2) =: T.
  \]

  Since, for $M \in \R^{2 \times 2}$, $\alpha_2 (M) = |\det(M)| / \alpha_1
  (M)$, we obtain from \eqref{eq:ineqalpha}  and the multiplicativity
  of the determinant that, for $\iii \in I^n$,
  \begin{equation}
    \label{eq:ineqfii} \lambda_1^n \fii^s (A_{\iii}) < \fii^s (B_{\iii}) <
    \lambda_2^n \fii^s (A_{\iii}),
  \end{equation}
  where
  \begin{align*}
    \lambda_1 &= (1 - \eps_1) (1 + \eps_1)^{- 1} (1 - \eps_2)^{T/2},\\
    \lambda_2 &= (1 + \eps_1) (1 - \eps_1)^{- 1} (1 + \eps_2)^{T/2} .
  \end{align*}
  In order to see that \eqref{eq:ineqfii} holds, it is convenient to
  consider the cases $0\leq s<1$, $1\leq s < 2$, and $2\leq s \leq T$
  separately. From \eqref{eq:ineqfii}, we obtain
  \[
  P_\PA(s) + \log (\lambda_1) \le P_\PB (s) \le P_\PA(s) + \log
  (\lambda_2),
  \]
  yielding
  \[
  P_\PA(s) \in [- \log (\lambda_2), - \log (\lambda_1)].
  \]
  Since $P_\PA$ is a continuous, strictly decreasing function, so is its
  inverse $P_\PA^{-1}$. But $\lambda_1, \lambda_2 \rightarrow 1$ as $\eps
  \rightarrow 0$, so the continuity of $P_\PA^{-1}$ implies that $s
  \to t$ as $\eps \downarrow 0$. This is exactly what we
  wanted to show.
\end{proof}

\section{Proof of the main result} \label{sec:proof}

This section is dedicated to the proof of Theorem
\ref{thm:main_result}. We first collect several lemmas which will
be used in the proof. These lemmas are geometric consequences of
Definition \ref{def:kakeya}. We remark that some of these lemmas
are analogous to results in \cite{HueterLalley1995}.

\begin{lemma} \label{thm:norm_alpha}
  Suppose that for each $i \in I$ there is a contractive invertible
  matrix $A_i \in \R^{2 \times 2}$ such that the conditions
  \eqref{K-coneinvariance} and \eqref{K-transpose} are satisfied. Then
  $|A_\iii x| \ge \cos(\beta) \alpha_1(A_\iii) |x|$ for all $x \in
  X(\theta,\beta)$ and all $\iii\in I^*$. Moreover,
  $\alpha_1(A_{\iii\jjj}) \ge \cos^2(\beta) \alpha_1(A_\iii)
  \alpha_1(A_\jjj)$ whenever $\iii,\jjj \in I^*$.
\end{lemma}

\begin{proof}
  Let $\iii \in I^*$, $x\in X(\theta,\beta)$,
  and write $x = x_1 \theta_1(A_\iii) +  x_2 \theta_2(A_\iii)$. We
  may assume that $|x|=1$. Since $\theta_1(A_\iii)$ is, by definition,
  the eigenvector of $A_\iii^* A_\iii$ corresponding
  to the largest eigenvalue, it follows from \eqref{K-coneinvariance},
  \eqref{K-transpose}, and the Perron-Frobenius Theorem
  that $\theta_1(A_\iii)\in X(\theta,\beta)$ and $|x_1| = x\cdot
  \theta_1(A_\iii) \ge \cos(\beta)$  (note that
  the Perron-Frobenius Theorem is usually stated for matrices preserving
  the positive cone, but it holds for any cone by a change of
  coordinates). Therefore
  \begin{equation*}
  |A_\iii x|^2 = | A_\iii^* A_\iii x \cdot x | = \alpha_1(A_\iii)^2 x_1^2
  +  \alpha_2(A_\iii)^2 x_2^2 \ge \alpha_1(A_\iii)^2 \cos^2(\beta)
  \end{equation*}
  giving the first claim.

  The second claim follows immediately since
  \begin{equation*}
    \alpha_1(A_{\iii\jjj}) \ge |A_{\iii}A_{\jjj}\theta|
    \ge \cos(\beta) \alpha_1(A_{\iii}) |A_{\jjj}\theta| \ge \cos^2(\beta)
    \alpha_1(A_\iii) \alpha_1(A_\jjj)
  \end{equation*}
  whenever $\iii,\jjj \in I^*$.
\end{proof}

\begin{remark} \label{P-semiconformal}
  Suppose that for each $i \in I$ there is a contractive invertible
  matrix $A_i \in \R^{2 \times 2}$ such that the conditions
  \eqref{K-coneinvariance} and \eqref{K-transpose} are satisfied.
  It follows immediately from Lemma \ref{thm:norm_alpha} that there
  exists a constant $D \ge 1$ such that for every $t\ge 0$
  \begin{equation*}
    D^{-1}\fii^t(A_\iii)\fii^t(A_\jjj) \le \fii^t(A_{\iii\jjj})
  \end{equation*}
  whenever $\iii,\jjj \in I^*$. In fact, $D=\cos^{-2}(\beta)$
  works.
\end{remark}

\begin{lemma} \label{thm:angle}
  Suppose that for each $i \in I$ there is a contractive invertible
  matrix $A_i \in \R^{2 \times 2}$ such that the conditions
  \eqref{K-coneinvariance} and \eqref{K-transpose} are satisfied. Then
  \newcounter{savedenumi}
  \begin{enumerate}
  \renewcommand{\labelenumi}{(\roman{enumi})}
  \renewcommand{\theenumi}{(\roman{enumi})}
  \item \label{atmost}
    the angle between the vectors $A_\iii\bigl( \theta_1(A_\iii)
    \bigr)$ and $A_\iii x$
    is at most a constant times $\alpha_2(A_\iii)/\alpha_1(A_\iii)$ for
    every $\iii \in I^*$ and $x\in X(\theta,\beta)$.
  \setcounter{savedenumi}{\value{enumi}}
  \end{enumerate}
  If in addition the condition \eqref{K-separation} is satisfied, then
  \begin{enumerate}
  \setcounter{enumi}{\value{savedenumi}}
  \renewcommand{\labelenumi}{(\roman{enumi})}
  \renewcommand{\theenumi}{(\roman{enumi})}
  \item \label{atleast}
    the angle between the vectors $A_\iii x$ and $A_\jjj y$ is at least
    a constant times $\alpha_2(A_{\iii \land \jjj})/\alpha_1(A_{\iii
    \land \jjj})$ for every $\iii,\jjj \in I^*$ and $x,y \in
    X(\theta,\beta)$.
  \end{enumerate}
\end{lemma}

\begin{proof}
  We first prove \ref{atmost}. Fix $\iii\in I^*$.
  Let $x \in S^1 \cap X(\theta,\beta)$ and denote by $\gamma$ the
  (smaller) angle between $A_\iii x$ and the major axis of the ellipse
  $A_\iii\bigl( B(0,1) \bigr)$, that is, the vector
  $A_\iii\bigl( \theta_1(A_\iii) \bigr)$. Since, by
  Lemma \ref{thm:norm_alpha}, we have $|A_\iii x| \ge
  \cos(\beta)\alpha_1(A_\iii)$, it follows that $|\sin(\gamma)| \le
  \alpha_2(A_\iii)/\bigl(\cos(\beta)\alpha_1(A_\iii)\bigr)$. We
  conclude
  \begin{equation*}
    |\gamma| \le \tfrac{\pi}{2}|\sin(\gamma)| \le \tfrac{\pi}{2\cos(\beta)}
    \alpha_2(A_\iii)/\alpha_1(A_\iii).
  \end{equation*}

  Next we show \ref{atleast}. Write $\iii=\kkk\iii'$ and
  $\jjj=\kkk\jjj'$, where $\kkk=\iii\wedge\jjj$, and notice that
  $\iii'$ and $\jjj'$ start with different symbols. Therefore it
  follows from \eqref{K-separation} that there exists a constant $c>0$
  (independent of $\iii$ and $\jjj$) such that the angle between
  $A_{\iii'}x$ and $A_{\jjj'}y$ is at least $c$ for any $x,y \in
  X(\theta,\beta)$. Hence it will be enough to prove the following
  claim: Given $c_1>0$ there is $c_2>0$ such that if $x,y \in
  S^1 \cap X(\theta,\beta)$ and
  $|x-y|\ge c_1$, then the angle between $A_\kkk x$ and $A_\kkk y$
  is at least $c_2 \alpha_2(A_\kkk)/\alpha_1(A_\kkk)$ for all
  $\kkk\in I^*$.

  To prove the claim consider the triangle with vertices $0,
  A_\kkk x, A_\kkk y$. Denote the angle at $0$ by $\gamma$. By Lemma
  \ref{thm:norm_alpha},
  the sides containing $0$ have lengths between
  $\cos(\beta) \alpha_1(A_\kkk)$ and $\alpha_1(A_\kkk)$, while by the
  assumption, the length of the third side is at least $c_1 \alpha_2(A_\kkk)
  $. We compute the area of the triangle in two ways. On the one
  hand, it is $|A_\kkk x||A_\kkk y|\sin(\gamma)/2 \le
  \alpha_1(A_\kkk)^2 \sin(\gamma)/2$. Since one of the other two angles
  of the triangle must be at least $\pi/6$ (otherwise $\gamma >
  2\pi/3$ and there is nothing to prove), the
  area of the triangle is also at least $\cos(\beta) c_1 \alpha_1(A_\kkk)
  \alpha_2(A_\kkk)\sin(\pi/6)/2$. By comparing these two estimates,
  the claim follows. The proof is complete.
\end{proof}

In \cite[\S 3]{HueterLalley1995}, it is claimed that
\eqref{K-coneinvariance} implies that the matrices $A_i$ are
strict contractions acting on the space of lines through the
origin with positive slope, where the metric is the smaller angle
between them. This assertion is wrong, as the following example
shows: let
\[
A = \left(%
\begin{array}{cc}
  1 & \eps \\
  \eps & \eps \\
\end{array}%
\right).
\]
Let $\ell$ be the line through the origin and $(\eps,1)$ and let
$\ell'$ be the line through the origin and $(2\eps,1)$. Then a
simple calculation shows that the angle between the lines $A\ell$
and $A\ell'$ is of the order of $\eps^{-1}$ times the angle
between $\ell$ and $\ell'$ as $\eps\downarrow 0$. However, the
next lemma, and in particular \eqref{eq:repeated_coneinv}, shows
that \cite[Proposition 3.1]{HueterLalley1995} is still correct.

\begin{lemma} \label{thm:eta_alpha}
  Suppose that for each $i \in I$ there is a contractive invertible
  matrix $A_i \in \R^{2 \times 2}$ such that the condition
  \eqref{K-coneinvariance} is satisfied. Then there exist constants
  $C \ge 1$ and $0 < \eta < 1$ such that
  \begin{equation*}
    \alpha_2(A_\iii) \le C \eta^{|\iii|}\alpha_1(A_\iii)
  \end{equation*}
  whenever $\iii \in I^*$.
\end{lemma}

\begin{proof}
  Let us first show that there exists $C_0 \ge 1$ and $0 < \eta < 1$
  such that
  \begin{equation} \label{eq:repeated_coneinv}
    A_\iii\bigl( X(\theta,\beta) \bigr) \subset
    X(A_\iii\theta/|A_\iii\theta|, C_0 \eta^{|\iii|}\beta)
  \end{equation}
  whenever $\iii \in I^*$.
  Denote the space of all lines through the origin which are
  contained in $\overline{X(\theta,\beta)}$ by
  $\mathcal{P}(\theta,\beta)$. The smaller angle between any two
  lines $\ell_1,\ell_2$ will be denoted by
  $\varangle(\ell_1,\ell_2)$. Since the maps $A_i$ are not necessarily
  contractions with respect to the metric $\varangle$, we will make
  use of a different, but equivalent, metric. This metric is used in
  some proofs of the Perron-Frobenius Theorem, see for example
  \cite[Lemma 3.4]{PollicottYuri1998}.

  Let $\ell_0$ be a line through the origin which is not contained
  in $X(\theta,\beta)$, and such that $\varangle(\ell_0,\ell)<\pi/2$
  for all $\ell\in\mathcal{P}(\theta,\beta)$. Define
  $d \colon \mathcal{P}(\theta,\beta)^2\rightarrow\R$ by setting
  \[
  d(\ell_1,\ell_2) =
  \bigl|\log\tan\bigl( \varangle(\ell_0,\ell_1) \bigr) - \log\tan\bigl(
  \varangle(\ell_0,\ell_2) \bigr)\bigr|
  \]
  as $\ell_1,\ell_2 \in \mathcal{P}(\theta,\beta)$.
  It is easy to verify that $d$ is indeed a metric and, moreover,
  there is a constant $C_0 \ge 1$ such that
  \begin{equation} \label{eq:equivalentmetrics}
    C_0^{-1/2} \varangle(\ell_1,\ell_2) \le d(\ell_1,\ell_2) \le C_0^{1/2}
    \varangle(\ell_1,\ell_2),
  \end{equation}
  for all $\ell_1,\ell_2\in\mathcal{P}(\theta,\beta)$.
  This is true since $\log\tan$ has a bounded derivative on a compact
  subset of $(0,\pi/2)$. We claim that
  the maps $A_i$ acting on $\mathcal{P}(\theta,\beta)$ are uniformly
  contractive with respect to $d$. To prove this, we may fix $i \in I$
  and assume that
  \[
  A_ i = \left(%
    \begin{array}{cc}
      a & b \\
      c & d \\
    \end{array}%
  \right).
  \]
  Moreover, after an appropriate rotation we can assume that
  $\ell_0$ is the $x$-axis, and all elements of
  $\mathcal{P}(\theta,\beta)$ have positive slope. Hence
  $a,b,c,d$ are nonzero and have the same sign. We will denote
  the slope of $\ell\in\mathcal{P}(\theta,\beta)$ by $s(\ell)$.
  After this normalization, we have
  \[
  d(A_i\ell_1,A_i\ell_2) = |\log\bigl(s(A_i\ell_1)\bigr) -
  \log\bigl(s(A_i\ell_2)\bigr)|,
  \]
  where
  \[
  s(A_i\ell) = \frac{c+ds(\ell)}{a+bs(\ell)}
  \]
  for any $\ell \in \mathcal{P}(\theta,\beta)$.
  In order to verify the claim, it suffices to show that the
  derivative of the function $g \colon \R \to \R$, $g(s) =
  \log\frac{c+de^s}{a+be^s}$, is strictly less than $1$ in absolute
  value. It is straightforward to see that
  \[
  |g'(s)| = \frac{|ad-bc|e^s}{(a+be^s)(c+de^s)}
  \]
  attains its maximum value at $s_0 = \tfrac12
  \log\tfrac{ac}{bd}$. Some elementary algebra shows that
  \begin{equation*}
    |g'(s_0)| = \frac{|ad-bc|}{ad+bc+2\sqrt{abcd}} < 1
  \end{equation*}
  which is exactly what we wanted.

  Using the claim and \eqref{eq:equivalentmetrics}, we see that there
  exists $0<\eta<1$ such that
  \[
  \varangle(A_\iii \ell_1, A_\iii \ell_2) < C_0 \eta^{|\iii|}
  \varangle(\ell_1,\ell_2),
  \]
  for any $\iii\in I^*$ and
  $\ell_1,\ell_2\in\mathcal{P}(\theta,\beta)$. Taking
  $\ell_1,\ell_2$ as the two lines which make up the boundary of
  $X(\theta,\beta)$, the assertion \eqref{eq:repeated_coneinv}
  follows.

  To finally prove the lemma, notice that for each $\iii \in I^*$,
  we have
  \begin{align*}
    \LL^2\bigl( A_\iii \bigl( B(0,1) \cap X(\theta,\beta) \bigr)
    \bigr) &= \LL^2\bigl( B(0,1) \cap X(\theta,\beta) \bigr)
    \det(A_\iii) \\ &= \beta \alpha_1(A_\iii) \alpha_2(A_\iii).
  \end{align*}
  On the other hand, using \eqref{eq:repeated_coneinv}, we have
  \begin{align*}
    \LL^2\bigl( A_\iii \bigl( B(0,1) \cap X(\theta,\beta) \bigr)
    \bigr) &\le \LL^2\bigl( B(0,\alpha_1(A_\iii)) \cap
    X(A_\iii\theta/|A_\iii\theta|,C_0\eta^{|\iii|}\beta) \bigr) \\ &=
    C \eta^{|\iii|}\beta \alpha_1(A_\iii)^2
  \end{align*}
  for some constant $C \ge 1$.
  Comparing the two last displayed formulas yields the result.
\end{proof}

Now we are ready to prove the main theorem.

\begin{proof}[Proof of Theorem \ref{thm:main_result}]
  The upper bound $\dimum(E) \le t$ holds in general,
  for example, see \cite{DouadyOesterle1980} and
  \cite{Falconer1988}. Since \eqref{K-projection} implies $\dimh(E)
  \ge 1$, it is enough to prove that $\dimlm(E) \ge t$. The
  continuity assertion will then follow from Lemma
  \ref{thm:continuity}.

  Recalling Remark \ref{P-semiconformal}, let
  $\mu$, $1 > \lambda_1(\mu) \ge \lambda_2(\mu) > 0$, and $c,D \ge 1$ be
  as in Theorem \ref{thm:semiconformal}. Fix $0 < \eps \le \tfrac12
  c^{-2}D^{-1}\lalpha \le \tfrac12$.
  Using Egorov's Theorem, we find an integer $n_0$ and a compact set
  $K \subset I^\infty$ so that $\mu(I^\infty \setminus K) < \eps$ and
  \begin{equation*}
    \lambda_k(\mu)^{n(1+\eps)} \le \alpha_k(A_{\iii|_n}) \le
    \lambda_k(\mu)^{n(1-\eps)}
  \end{equation*}
  whenever $\iii \in K$, $k \in \{ 1,2 \}$, and $n \ge n_0$.
  Denoting
  \begin{equation*}
    \gamma_k = \frac{\log\lambda_k(\mu)}{\log\lambda_1(\mu)\lambda_2(\mu)^{t-1}}
  \end{equation*}
  as $k \in \{ 1,2 \}$, we notice that $\gamma_1 + (t-1)\gamma_2 = 1$
  and
  \begin{equation} \label{eq:alpha_approx}
    \fii^t(A_{\iii|_n})^{\gamma_k(1+\eps)/(1-\eps)} \le
    \alpha_k(A_{\iii|_n}) \le \fii^t(A_{\iii|_n})^{\gamma_k(1-\eps)/(1+\eps)}
  \end{equation}
  whenever $\iii \in K$, $k \in \{ 1,2 \}$, and $n \ge n_0$. Since
  $\eps$ can be arbitrarily small, \eqref{eq:alpha_approx} together
  with Lemma \ref{thm:eta_alpha} imply that $\gamma_1 < \gamma_2$.

  For $r > 0$ define
  \begin{equation*}
    Z(r) = \{ \iii \in I^* : \fii^t(A_\iii) \le r < \fii^t(A_{\iii^-})
    \},
  \end{equation*}
  and notice that the set $Z(r)$ is incomparable for every $r>0$.
  Denote also $Z_K(r) = \{ \iii \in Z(r) : [\iii] \cap K \ne \emptyset
  \}$. Since
  \begin{align*}
    (cD)^{-1}\lalpha\sum_{\iii \in Z(r) \setminus Z_K(r)} r &\le
    c^{-1}\sum_{\iii \in Z(r) \setminus Z_K(r)} \fii^t(A_\iii) \\ &\le
    \sum_{\iii \in Z(r) \setminus Z_K(r)} \mu([\iii]) \le \mu(I^\infty
    \setminus K) < \eps,
  \end{align*}
  and, similarly,
  \begin{equation*}
    \sum_{\iii \in Z(r)} r \ge \sum_{\iii \in Z(r)} \fii^t(A_\iii) \ge
    c^{-1},
  \end{equation*}
  it follows that
  \begin{equation} \label{eq:zetacardinality}
    \# Z_K(r) \ge (c^{-1} - \eps cD\lalpha^{-1})r^{-1} \ge \tfrac12
    c^{-1} r^{-1}.
  \end{equation}
  Hence, choosing $r>0$ small enough so that $|\iii| \ge n_0$ for
  every $\iii \in Z(r)$ and denoting $\xi = \min_{k \in \{ 1,2 \}}
  (D^{-1}\lalpha)^{3\gamma_k}$, it follows from
  \eqref{eq:alpha_approx} that
  \begin{equation} \label{eq:alpha_approx2}
    \xi r^{\gamma_k(1+4\eps)} \le
    \alpha_k(A_\iii)
    \le r^{\gamma_k(1-2\eps)}
  \end{equation}
  whenever $\iii \in Z_K(r)$ and $k \in \{ 1,2 \}$.

  Fix $\iii \in I^*$. Let $v_\iii$ be the unit vector with direction
  equal to the major axis of the ellipse $A_\iii\bigl(B(0,1)\bigr)$.
  Explicitly, $v_\iii = A_\iii\bigl( \theta_1(A_\iii) \bigr)/\alpha_1(A_\iii)$.
  Since
  \[
  v_\iii \cdot A_\iii x = A_\iii^* v_\iii \cdot x = \alpha_1(A_\iii)
  \theta_1(A_\iii)\cdot x,
  \]
  for each $x \in E$,
  it follows from \eqref{K-projection} that
  $\LL^1\bigl( \{ v_\iii \cdot x : x \in E_\iii \}
  \bigr) \ge \roo\alpha_1(A_\iii)$. Hence there exists a
  constant $T \ge 1$ so that for each $\iii \in I^*$ there is a
  rectangle $R_\iii$ of size
  $\alpha_1(A_\iii)\times\alpha_2(A_\iii)$ with long side parallel
  to $A_\iii\bigl( \theta_1(A_\iii) \bigr)$ such that the
  $T\alpha_2(A_\iii)$-neighborhood of $E_\iii$ intersects
  $R_\iii$ in a set of $\LL^2$-measure at least
  $\roo\alpha_1(A_\iii)\alpha_2(A_\iii)$.

  Using Lemma \ref{thm:angle}\ref{atleast} and
  (\ref{eq:alpha_approx}), we get that there exists a constant
  $0<\omega'<1$ such that if $\iii,\jjj\in Z_K(r)$, $\iii\neq\jjj$, and
  $|\iii\land\jjj|\ge n_0$, then the
  angle between the long sides of the rectangles
  $R_\iii$ and $R_\jjj$, denoted by $\varangle(R_\iii,R_\jjj)$, is at
  least
  \begin{equation*}
    \omega' \alpha_2(A_{\iii \land \jjj})/\alpha_1(A_{\iii \land \jjj})
    \ge \frac{\omega' \fii^t(A_{\iii \land
        \jjj})^{\gamma_2(1+\eps)/(1-\eps)}}{\fii^t(A_{\iii \land
        \jjj})^{\gamma_1(1-\eps)/(1+\eps)}}
    \ge \omega' r^{\gamma_2(1+4\eps)-\gamma_1(1-2\eps)}.
  \end{equation*}
  If $|\iii\land\jjj|<n_0$ then, using Lemma
  \ref{thm:angle}\ref{atleast} again,
  \[
  \varangle(R_\iii,R_\jjj) \ge \omega'  \alpha_2(A_{\iii \land
    \jjj})/\alpha_1(A_{\iii \land \jjj}) \ge \omega'
  \lalpha^{n_0}/\ualpha^{n_0}.
  \]
  Thus, in either case, if $\iii,\jjj\in Z_K(r)$, $\iii\neq\jjj$,
  then
  \begin{equation}   \label{eq:anglelowerbound}
  \varangle(R_\iii,R_\jjj) \ge \omega r^{\gamma_2(1+4\eps)-\gamma_1(1-2\eps)},
  \end{equation}
  where $\omega = \omega' \lalpha^{n_0}/\ualpha^{n_0} < 1$.

  In order to apply Proposition \ref{thm:kakeya_estimate}, all the
  rectangles must have the same size. Let
  \[
  \alpha_1' = r^{\gamma_1(1-2\eps)}
  \frac{r^{\gamma_2(1-2\eps)}}{\omega r^{\gamma_2(1+4\eps)}} \ge
  r^{\gamma_1(1-2\eps)}
  \]
  and let also $\alpha_2' = r^{\gamma_2(1-2\eps)}$ (bear in mind that
  both $\alpha_1'$ and $\alpha_2'$ depend on $r$).
  It follows from \eqref{eq:alpha_approx2} that each rectangle
  $R_\iii$, with $\iii\in Z_K(r)$, is
  contained in a rectangle $R_\iii'$ of size $\alpha_1' \times
  \alpha_2'$ with long side still parallel to $A_\iii\bigl(
  \theta_1(A_\iii) \bigr)$. Moreover, by \eqref{eq:anglelowerbound}, the
  angle between any two such rectangles is at least
  $\alpha_2'/\alpha_1'$.

  Let $\delta =  T r^{\gamma_2 (1-2\eps)}$. We write
  $E(\delta)$ for the $\delta$-neighborhood of $E$. Using
  \eqref{eq:alpha_approx2} once again, notice that, whenever $\iii\in
  Z_K(r)$, $E(\delta)$
  contains a $T\alpha_2(A_\iii)$-neighborhood of $E_\iii\subset E$.
  Hence $E(\delta)$ intersects each rectangle $R_\iii$, and
  therefore also each rectangle $R_\iii'$, in a set of
  $\LL^2$-measure at least
  \[
  \roo\alpha_1(A_\iii)\alpha_2(A_\iii) \ge \roo\xi
  r^{\gamma_1(1+4\eps)}\xi
  r^{\gamma_2(1+4\eps)} =  \tau
  \alpha_1'\alpha_2',
  \]
  where
  \begin{equation*}
    \tau = \roo\xi^2 r^{(\gamma_1+\gamma_2){(1+4\eps)}}
    \frac{\omega r^{\gamma_2(1+4\eps)}}{r^{\gamma_1(1-2\eps)}
      r^{\gamma_2(1-2\eps)}r^{\gamma_2(1-2\eps)}}
    = \roo \xi^2 \omega r^{6\eps(\gamma_1+2\gamma_2)}.
  \end{equation*}

  We can now apply Proposition \ref{thm:kakeya_estimate} to the
  set $E(\delta)$ and the family $\{ R_\iii' : \iii \in Z_K(r)
  \}$ to obtain, for every $r>0$ small enough, that
  \begin{align*}
    \LL^2\bigl(E(\delta)\bigr) &\ge \frac{\# Z_K(r)\tau
      (\tau\alpha_1'\alpha_2')}
         {2\sqrt{2}\log(2\pi\alpha_1'/\alpha_2')} \\
    &\ge  \frac{\left(\tfrac12 c^{-1}r^{-1}\right)\left(\roo \xi^2
        \omega r^{6\eps(\gamma_1+2\gamma_2)}\right)
      \left(\roo\xi^2r^{(\gamma_1+\gamma_2)(1+4\eps)}\right)}
         {2\sqrt{2}\log(2\pi\omega^{-1} r^{\gamma_1(1-2\eps)-\gamma_2(1+4\eps)})} \\
    &= \frac{(4\sqrt{2}c)^{-1}\roo^2\xi^4 \omega
         r^{\gamma_1+\gamma_2-1+\eps(10\gamma_1+16\gamma_2)}}
         {\log(2\pi\omega^{-1} r^{\gamma_1-\gamma_2-\eps(2\gamma_1+4\gamma_2)}) },
  \end{align*}
  where in the second displayed line we used
  (\ref{eq:zetacardinality}). Recalling the definition of $\delta$, we
  estimate
  \begin{align*}
    \dimlm(E) &= \liminf_{\delta \downarrow 0} \biggl( 2 -
    \frac{\log\LL^2\bigl(E(\delta)\bigr)}{\log\delta} \biggr) \\
    &\ge 2 - \limsup_{r \downarrow 0} \biggl(
    \frac{\log\bigl((4\sqrt{2}c)^{-1}\roo^2\xi^4\omega
         r^{\gamma_1+\gamma_2-1+\eps(10\gamma_1+16\gamma_2)}\bigr)}
         {\log(T r^{\gamma_2 (1-2\eps)})} \\
    &\qquad\qquad\quad\;\;\, - \frac{\log\log(2\pi\omega^{-1}
      r^{\gamma_1-\gamma_2-\eps(2\gamma_1+4\gamma_2)}) }
         {\log(T r^{\gamma_2 (1-2\eps)})} \biggr) \\
    &= 2 - \frac{\gamma_1+\gamma_2-1+\eps(10\gamma_1+16\gamma_2)}
         {\gamma_2(1-2\eps)},
  \end{align*}
  provided that $\gamma_1-\gamma_2-\eps(2\gamma_1+4\gamma_2) < 0$. By our
  earlier remark that $\gamma_1 < \gamma_2$, this can be achieved by
  starting with a very small $\eps>0$. Since $\gamma_1-1 =
  (1-t)\gamma_2$, we conclude, by letting $\eps\downarrow 0$, that
  \[
  \dimlm(E) \ge 2 - \frac{\gamma_1+\gamma_2-1}{\gamma_2} = t,
  \]
  as desired.
\end{proof}

\section{On the projection condition} \label{sec:projection}

Of all the conditions in the definition of a self-affine set of Kakeya
type, the projection condition \eqref{K-projection} is the only
one which cannot be checked directly. In this section we prove
easily verifiable criteria which will be used to produce examples
where \eqref{K-projection} holds.

We introduce some notation. Given a set $F\subset\R^d$ and
$e\in\R^d$, we will denote
\[
F \cdot e = \{ x\cdot e: x\in F\}.
\]
The convex hull of $F$ will be denoted by $\conv(F)$. Recall that
a matrix $M\in\R^{\kappa\times\kappa}$ with nonnegative
coefficients is \emph{irreducible} if for all $1\le i,j\le\kappa$
there is $n>0$ such that $M^n_{ij}>0$. Finally, the identity
matrix on $\R^{2\times 2}$ will be denoted by $\id_2$.

We state two simple lemmas for later reference.

\begin{lemma} \label{th:intersectionisinterval}
  If $\{ \intI_\iii : \iii \in I^* \}$ is a collection of closed
  intervals such that for any $\iii \in I^*$
  \[
  \bigcup_{i \in I} \intI_{\iii i} = \intI_\iii,
  \]
  then
  \[
  \bigcap_{k=0}^\infty \bigcup_{\iii \in I^k} \intI_\iii =
  \intI_\varnothing.
  \]
\end{lemma}

\begin{proof}
  Immediate by induction.
\end{proof}

\begin{lemma} \label{th:unionofintervals}
  Suppose $\intI_1,\ldots, \intI_\kappa$ are closed intervals.
  If the adjacency matrix $M \in \R^{\kappa \times \kappa}$ defined as
  \begin{equation*}
    M_{ij} =
    \begin{cases}
      1, &\text{if } \intI_i \cap \intI_j \ne \emptyset,\\
      0, &\text{otherwise,}
    \end{cases}
  \end{equation*}
  is irreducible, then $\bigcup_{i=1}^\kappa \intI_i$ is an interval.
\end{lemma}

\begin{proof}
  Left to the reader.
\end{proof}

The following proposition, which may be of independent interest,
provides a simple criterion to guarantee that all the projections
of a self-affine set are intervals. Even though our application
will be in $\R^2$, we state the result for affine IFS's on
$\R^\kappa$ since the proof is the same.

\begin{proposition} \label{thm:convexprojectioncondition}
  Suppose that for each $i\in I$ there are a contractive invertible
  matrix $A_i \in \R^{\kappa \times \kappa}$ with $||A_i|| \le \ualpha<1$ and a
  translation vector $a_i\in\R^\kappa$. Assume the adjacency matrix $M \in
  \R^{\kappa \times \kappa}$ defined as
  \begin{equation*}
    M_{ij} =
    \begin{cases}
      1, &\text{if } \conv(E_i) \cap \conv(E_j) \neq \emptyset,\\
      0, &\text{otherwise},
    \end{cases}
  \end{equation*}
  is irreducible. Then $E \cdot e = \conv(E) \cdot e$ for all $e \in
  \R^\kappa$ and, in particular, $E \cdot e$ is an interval or a single
  point.
\end{proposition}

\begin{proof}
  We will repeatedly use the fact that the action of taking convex
  hulls commutes with affine maps. As a first instance of this,
  observe that for any $\iii \in I^*$,
  \begin{equation} \label{eq:affineconvex}
    A_\iii \bigl( \conv(E) \bigr) +  a_\iii = \conv(E_\iii),
  \end{equation}
  where
  \begin{equation} \label{eq:a_i}
    a_\iii = \sum_{n=1}^{|\iii|} A_{\iii|_{n-1}}a_{i_n}.
  \end{equation}
  Let $D$ denote the Hausdorff distance. Notice that
  \eqref{eq:affineconvex} implies
  \[
  D\bigl(\conv(E_\iii), E_\iii\bigr) \le \alpha_1(A_\iii)
  D\bigl(\conv(E), E\bigr).
  \]
  Therefore
  \[
  \lim_{k\rightarrow\infty} D\biggl(\bigcup_{\iii \in I^k}
  \conv(E_\iii), E\biggr) = 0,
  \]
  which in turn yields that
  \[
  E \cdot e = \bigcap_{k=1}^\infty \bigcup_{\iii \in I^k}
  \conv(E_\iii) \cdot e.
  \]
  Hence in order to prove the proposition it is enough to show that
  the family $\{ \conv(E_\iii) \cdot e:\iii\in I^*\}$
  verifies the hypothesis of Lemma \ref{th:intersectionisinterval}
  for all $e \in \R^\kappa$. We will do so by induction on
  $|\iii|$. Denote $\intI_\iii = \conv(E_\iii) \cdot
  e$ as $\iii \in I^*$, and note that $\intI_i \cap \intI_j \neq
  \emptyset$ whenever $\conv(E_i) \cap \conv(E_j) \neq
  \emptyset$. Since the matrix $M$ was assumed to be irreducible, the
  hypothesis of Lemma \ref{th:unionofintervals} is met, whence
  $\intJ_\varnothing := \bigcup_{i\in I} \intI_i$ is an interval, and
  thus equal to its convex hull. On the other hand, since
  \[
  E \cdot e \subset \intJ_\varnothing \subset \conv(E) \cdot e =
  \conv(E \cdot e),
  \]
  we have $\conv(\intJ_\varnothing)= \conv(E) \cdot e$. Hence
  $\intJ_\varnothing = \conv(E) \cdot e$, and this settles the case
  $|\iii|=0$. Now assume the case $|\iii|=k$ has been proven, and
  let $\iii$ be a symbol of length $k+1$. Write $\intJ_\iii = \bigcup_{i \in I}
  \conv(E_{\iii i}) \cdot e$ and $\iii = j\jjj$, where
  $j\in I$ and $|\jjj|=k$. Then
  \begin{align*}
    \intJ_\iii &= \bigcup_{i \in I} \bigl( A_j \bigl( \conv(E_{\jjj
      i}) \bigr) + a_j \bigr) \cdot e
    = A_j \biggl( \bigcup_{i \in I} \conv(E_{\jjj i}) \biggr) \cdot e + a_j
    \cdot e \\
    &= \biggl( \bigcup_{i \in I} \conv(E_{\jjj i}) \biggr) \cdot A_j^* e +
    a_j \cdot e.
  \end{align*}
  By the inductive hypothesis, this is an interval. On the other
  hand, $\intJ_\iii$ contains $E_\iii \cdot e$ and is contained in
  $\conv(E_\iii) \cdot e$, whence its convex hull must be
  $\conv(E_\iii) \cdot e$. This shows that $\intJ_\iii = \intI_\iii$,
  which is what we wanted to prove.
\end{proof}

Proposition \ref{thm:convexprojectioncondition} is useful because
one can check whether it holds by simply plotting the self-affine
set $E$, say using a computer program. It also yields a very
simple algebraic criterion which guarantees that all linear
projections are stably intervals, as the next corollary shows.
Given $x,y \in \R^2$, we will denote $[x,y] = \{ \lambda x +
(1-\lambda)y : 0 \le \lambda \le 1 \}$ and $(x,y) = [x,y]
\setminus \{ x,y \}$. Furthermore, if $i \in I$ then with the
notation $i^\infty$, we mean the symbol $(i,i,\ldots) \in
I^\infty$.

\begin{corollary} \label{thm:intersectionprojectioncondition}
  Suppose that for each $i\in I$ there are a contractive invertible
  matrix $A_i \in \R^{2 \times 2}$ with $||A_i|| \le \ualpha$ and a
  translation vector $a_i\in\R^2$. Denote by $E$ the invariant set
  of the affine IFS $\Phi = \{ A_i + a_i\}_{i \in I}$ and let
  \begin{equation} \label{eq:x_i}
  \begin{split}
    x_i &= \pi(i 1^\infty) = a_i + \sum_{n=0}^\infty A_i A_1^n a_1, \\
    y_i &= \pi(i \kappa^\infty) = a_i + \sum_{n=0}^\infty A_i
    A_\kappa^n a_\kappa,
  \end{split}
  \end{equation}
  as $i \in I$.
  If the adjacency matrix $M \in \R^{\kappa \times \kappa}$ defined as
  \begin{equation*}
    M_{ij} =
    \begin{cases}
      1, &\text{if $(x_i,y_i) \cap (x_j,y_j)$ is a single point}, \\
      0, &\text{otherwise},
    \end{cases}
  \end{equation*}
  is irreducible, then for each affine IFS $\Phi'$ sufficiently close
  to $\Phi$ there is a constant $\roo > 0$ such that $E' \cdot e$ is
  an interval having length at least $\roo$ for all $e \in \R^2$.
  Here $E'$ is the invariant set of $\Phi'$.
\end{corollary}

\begin{proof}
  Denote by $M'$ the adjacency matrix corresponding to the
  system $\Phi'$. Since the property that $(x_i,y_i)$
  intersects $(x_j,y_j)$ in a single point is stable, we see
  that $M'_{ij} \ge M_{ij}$ if $\Phi'$ is
  sufficiently close to $\Phi$. In particular, $M'$ is
  irreducible whenever $M$ is. Thus it is enough to verify the result
  for the original system $\Phi$. It follows from the assumptions that
  $E$ is not contained in a line. Thus there exists $\roo>0$ such that
  $\conv(E)$ contains a ball of radius $\roo$. Since trivially
  $(x_i,y_i) \subset \conv(E_i)$, the proof is finished by Proposition
  \ref{thm:convexprojectioncondition}.
\end{proof}

We next present a different, but also stable and easily checkable,
condition that guarantees that the projection condition
\eqref{K-projection} is met. Let $\QQ_2$ denote the family of all
vectors $v \in \R^2$ with strictly positive coefficients and define
a partial order $\prec$ on $\R^2$ by setting $x \prec y$ if
and only if $y-x \in \QQ_2$. With the notation $x \preceq y$ we mean
that $x \prec y$ or $x=y$.

\begin{lemma} \label{thm:orderprojectioncondition}
  Suppose that for each $i\in I$ there are a contractive invertible
  matrix $A_i \in \R^{2 \times 2}$ with $||A_i|| \le \ualpha$ and a
  translation vector $a_i\in\R^2$. If $A_i$ has strictly positive
  coefficients for all $i\in I$ and the points $x_i$, $y_i$ defined
  in \eqref{eq:x_i} satisfy
  \begin{equation} \label{eq:chain}
    x_i \prec x_{i+1} \prec y_i \prec y_{i+1}
  \end{equation}
  whenever $i \in \{ 1,\ldots,\kappa-1 \}$, then there is a constant
  $\roo>0$ such that $E \cdot e$ contains an interval of length
  $(y_\kappa-x_1)\cdot e \ge \roo$ for all $e \in \QQ_2$.
\end{lemma}

\begin{proof}
  The proof runs parallel to that of Proposition
  \ref{thm:convexprojectioncondition}. Given $\iii\in I^*$, write
  \[
  \ell_\iii = A_\iii([x_1,y_\kappa])+a_\iii,
  \]
  where $a_\iii$ is given by \eqref{eq:a_i}. We set
  $A_\varnothing = \id_2$ and $a_\varnothing=(0,0)$. Observe that
  \[
  D(\ell_\iii, E_\iii) \rightarrow 0 \textrm{ as }
  |\iii|\rightarrow\infty,
  \]
  whence
  \[
  \lim_{k\rightarrow\infty} D\biggl(\bigcup_{\iii \in I^k} \ell_\iii, E\biggr) = 0,
  \]
  which in turn yields that
  \[
  E \cdot e \supset \bigcap_{k=1}^\infty \bigcup_{\iii \in I^k}
  \ell_\iii \cdot e.
  \]
  Thus we only need to prove that the family $\{ \ell_\iii \cdot
  e:\iii\in I^*\}$ verifies the hypothesis of Lemma
  \ref{th:intersectionisinterval} for all $e\in\QQ_2$. Denoting
  $\intI_\iii = \ell_\iii \cdot e$ as $\iii \in I^*$, we
  will prove by induction on $|\iii|$ that
  \begin{equation} \label{eq:inductionunionints}
    \intI_\iii = \bigcup_{i\in I}
    \intI_{\iii i} = \bigl(A_\iii([x_1,y_\kappa]) + a_\iii\bigr) \cdot
    e,
  \end{equation}
  for all $e \in \QQ_2$. Consider the case $|\iii| = 0$ first. Note
  that, for $i \in I$,
  \[
  x_i = A_i x_1 + a_i, \quad y_i = A_i y_\kappa + a_i.
  \]
  Hence $\ell_i = [x_i,y_i]$. From \eqref{eq:chain} we get that $x_1
  \preceq x_i \prec y_i \preceq y_\kappa$ for $i\in I$, whence
  \begin{equation} \label{eq:unionints1}
    \bigcup_{i\in I} [x_i,y_i]\cdot e \subset [x_1,y_\kappa]\cdot e.
  \end{equation}
  On the other hand, from \eqref{eq:chain} we see that $x_i \prec
  y_{i+1}$ and $x_{i+1}\prec y_i$. Since $x\cdot e < y\cdot e$
  whenever $x\prec y$ and $e\in\QQ_2$, we get
  \begin{equation} \label{eq:unionints2}
    [x_i,y_i] \cdot e \cap [x_{i+1},y_{i+1}]\cdot e \neq \emptyset,
  \end{equation}
  whenever $i \in \{ 1,\ldots,\kappa-1 \}$. From \eqref{eq:unionints1} and
  \eqref{eq:unionints2}, and recalling that $\ell_i = [x_i,y_i]$, we
  get \eqref{eq:inductionunionints} in the case $|\iii|=0$. The inductive
  step follows the same pattern as in Proposition
  \ref{thm:convexprojectioncondition}; details are omitted.
\end{proof}

\section{Examples and remarks} \label{sec:examples}

We are now ready to state easily checkable conditions which
guarantee that an affine IFS is stably of Kakeya type. Explicit
examples follow below. In the following theorem, we will use the
convention that $[x,y]=[y,x]$ if $x>y$.

\begin{theorem} \label{thm:example}
  Suppose that for each $i \in I$ there are a contractive invertible
  matrix $A_i \in \R^{2 \times 2}$ with $||A_i|| \le \ualpha$ and a
  translation vector $a_i \in \R^2$. Assume further that for each $i
  \in I$ there are real numbers $u_i,v_i,w_i,z_i>0$ such that
  \[
  A_i = \left(%
    \begin{array}{cc}
      u_i & v_i \\
      w_i & z_i \\
    \end{array}%
  \right)
  \]
  and the following two conditions hold:
  \begin{enumerate}
  \renewcommand{\labelenumi}{(X\arabic{enumi})}
  \renewcommand{\theenumi}{X\arabic{enumi}}
  \item \label{X-separation}
    The intervals $[w_i/u_i,z_i/v_i]$ are pairwise disjoint for every
    $i \in I$.
  \item \label{X-projection}
    The affine IFS $\{A_i+a_i\}_{i \in J}$, where $J \subset I$ has
    cardinality at least $2$, verifies either the
    hypotheses of Corollary \ref{thm:intersectionprojectioncondition}
    or the hypotheses of Lemma \ref{thm:orderprojectioncondition}.
  \end{enumerate}
  Then the affine IFS $\Phi=\{ A_i+a_i\}_{i \in I}$ is stably of
  Kakeya type. In particular, the Minkowski dimension of the invariant
  set is given by the zero of the pressure formula
  \eqref{eq:pressure}, and is continuous on a neighborhood of $\Phi$.
\end{theorem}

\begin{proof}
  Using Theorem \ref{thm:main_result}, we only need to show that
  \eqref{K-kakeya} and \eqref{K-projection} hold for any small
  perturbation of $\Phi$. Since both \eqref{X-separation} and
  \eqref{X-projection} are stable properties, it is in fact enough to
  check that $\Phi$ is of Kakeya type.

  Let $\theta=\frac{1}{\sqrt{2}}(1,1)$. Since the $A_i$ have
  strictly positive coefficients, both $A_i$ and $A_i^*$ map the
  cone $X(\theta,\pi/2)$ into $X(\theta,\beta')$ for some
  $\beta'<\pi/2$. Hence there exists $\beta<\pi/2$ such that both
  \eqref{K-coneinvariance} and \eqref{K-transpose} hold.

  Suppose that \eqref{K-separation} does not hold for $\Phi$. Then
  there is $s>0$ and $i,j \in I$ such that $i \ne j$ and
  \[
  (1,s) = A_i(x,y) = A_j(x',y'),
  \]
  for some $x,y,x',y'>0$. Some simple algebra shows that
  \[
  s = \frac{w_i x+z_i y}{u_i x + v_i y} = \frac{w_j x'+z_j y'}{u_j
    x' + v_j y'},
  \]
  whence $s\in [w_i/u_i,z_i/v_i]\cap [w_j/u_j, z_j/v_j]$, which
  contradicts \eqref{X-separation}.

  Let $F$ be the invariant set of $\Psi=\{ A_i
  +a_i\}_{i \in J}$. It is clear that $F \subset E$. If $\Psi$
  verifies the conditions of Corollary
  \ref{thm:intersectionprojectioncondition}, then \eqref{K-projection}
  is immediately satisfied for $\Psi$ and hence also for $\Phi$.
  Likewise, if $\Psi$ satisfies the hypotheses of Lemma
  \ref{thm:orderprojectioncondition}, then \eqref{K-projection} holds
  for $\Phi$. This is true since $\theta_1(A_i)$ has positive
  coordinates thanks to the Perron-Frobenius Theorem.
  The proof is complete.
\end{proof}

We remark that finding an explicit neighborhood to which Theorem
\ref{thm:example} applies is an elementary, if tedious, exercise.

\begin{example} \label{ex:intersection}
  We consider our first specific example. Let
  \[
  A_1(r,\eps) = \left(%
    \begin{array}{cc}
      r & r+\eps \\
      \eps & r \\
    \end{array}%
  \right), \quad
  A_2(r,\eps) = \left(%
    \begin{array}{cc}
      r & \eps \\
      r+\eps & r \\
    \end{array}%
  \right).
  \]

  \begin{figure}
    \centering
    \includegraphics[width=0.6\textwidth]{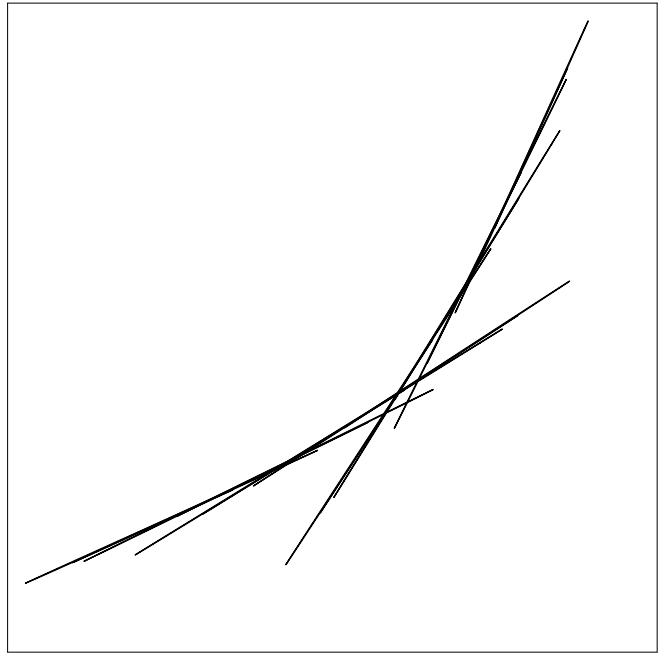}
    \caption{A self-affine set of Kakeya type.} \label{fig-sa}
  \end{figure}

  The affine IFS $\{A_1(r,0)+a_1, A_2(r,0)+a_2\}$ was studied in
  \cite{Edgar1992}, where it is proven that the singularity
  dimension is $1$ when $r=1/3$. This IFS does not verify
  \eqref{K-coneinvariance}; however, $\{A_1(r,\eps)+a_1,
  A_1(r,\eps)+a_2\}$ does satisfy \eqref{X-separation}, and hence
  \eqref{K-kakeya}, for all small $\eps>0$.

  Figure \ref{fig-sa} depicts the invariant set when $r=0.4$,
  $\varepsilon = 0.1$ and the translations are $a_1 = (-0.3, -0.3)$
  and $a_2=-a_1$. For these values of the parameters the spectral
  radius of the matrices $A_i(r,\eps)$ is approximately $0.624>1/2$; thus
  Falconer's Theorem does not apply. However, the conditions of
  Corollary \ref{thm:intersectionprojectioncondition} are clearly
  met (this can be verified algebraically without effort). Thus, by
  Theorem \ref{thm:example}, this is stably a self-affine set of
  Kakeya-type. We remark that by picking appropriate values of $r$
  and $\eps$ one can obtain examples where the norms of the maps are
  arbitrarily close to $1$.

  Notice that the invariant set resembles a union of approximately
  equally long segments pointing in different directions,
  underlining the Kakeya-type structure. Also observe that this
  particular example appears to be overlapping, although proving
  this rigorously looks very difficult.
\end{example}

\begin{lemma} \label{thm:exampleorder}
  Suppose that for each $i \in \{1,2\}$ there is a contractive invertible
  matrix $A_i\in\R^{2\times 2}$ with strictly positive coefficients and $||A_i||
  \le \ualpha$, such that the condition \eqref{X-separation} is
  satisfied. Let
  \begin{equation*}
    B_2 = \sum_{n=0}^\infty A_2^n = (\id_2-A_2)^{-1}.
  \end{equation*}
  If both $A_1 B_2 - \id_2$ and $(\id_2-A_1)B_2$
  have strictly positive coefficients, then for any
  vector $a_2$ with strictly positive coefficients, the affine IFS $\{ A_1,
  A_2+a_2\}$ is stably of Kakeya type.
\end{lemma}

\begin{proof}
  Notice that $B_2$ has strictly positive coefficients. The points
  defined in \eqref{eq:x_i} are now
  $x_1 = 0$, $y_1 = A_1 B_2 a_2$, $x_2 = a_2$, and $y_2 = B_2 a_2$. Suppose
  $a_2\in\QQ_2$. It is clear that $x_1\prec x_2$. Moreover, $x_2
  \prec y_1$ whenever $A_1 B_2 - \id_2$ has strictly positive coefficients,
  and $y_1 \prec y_2$ whenever $(\id_2-A_1)B_2$ has strictly positive
  coefficients. Thus we have shown that the hypotheses of Lemma
  \ref{thm:orderprojectioncondition} hold, whence the lemma is
  immediate from Theorem \ref{thm:example}.
\end{proof}

\begin{example} \label{ex:concreteexample}
  As a concrete example, let
  \begin{equation} \label{eq:concreteexample}
    A_1 = \left(%
      \begin{array}{cc}
        0.35 & 0.40 \\
        0.30 & 0.35 \\
      \end{array}%
    \right), \quad
    A_2 = \left(%
      \begin{array}{cc}
        0.40 & 0.45 \\
        0.45 & 0.50 \\
      \end{array}%
    \right) .
  \end{equation}
  A straightforward calculation shows that $A_1 B_2 -\id_2$ and
  $(\id_2-A_1)B_2$ have positive coefficients. Hence, by Lemma
  \ref{thm:exampleorder}, the affine IFS $\{ A_1, A_2 + a_2\}$, as well as
  any small perturbation, is of Kakeya type for any $a_2\in\QQ_2$.
  In particular, the Minkowski dimension of the invariant set of
  this IFS is constant for all $a_2\in\QQ_2$.
\end{example}

\begin{example}
  As a final example, we exhibit an affine IFS of Kakeya type with
  an arbitrary number of maps. Choose $\kappa \ge 3$ and let $A_1$,
  $A_2$ be as in Example \ref{ex:concreteexample}. For $j \in \{
  3,\ldots,\kappa \}$, we define
  \[
  A_j = \left(%
    \begin{array}{cc}
      \tfrac{1}{2} & \tfrac{1}{2} \\
      \tfrac{1}{3j-1} & \tfrac{1}{3j} \\
    \end{array}%
  \right).
  \]
  Note that $\{ A_1,\ldots, A_\kappa\}$ satisfies
  \eqref{X-separation}. Thus Theorem \ref{thm:example}, applied with
  $J=\{1,2\}$, implies that for any $a_2\in\QQ_2$ and any
  $a_3,\ldots,a_\kappa\in\R^2$, the affine IFS
  \[
  \{ A_1, A_2+a_2, A_3+a_3,\ldots, A_\kappa+a_\kappa\}
  \]
  is stably of Kakeya type.
\end{example}

We finish the paper with some questions and remarks.

\begin{remark} \label{remarks}
  (1) Our
  techniques do not extend easily to higher dimensions. One source
  of technical difficulties is having to deal with more than two
  singular values, but the main obstruction is of course that the
  Kakeya conjecture is open for dimension $d\ge 3$, and no analogue
  of Proposition \ref{thm:kakeya_estimate} is known. We remark,
  however, that Lemma \ref{thm:norm_alpha} does hold, with the same
  proof, in higher dimensions, although one needs to replace the
  cone $X(\theta,\beta)$ by a cone which is, after a change of
  coordinates, $\QQ_d \cup -\QQ_d$. Here $\QQ_d$ is the family of all
  vectors $v \in \R^d$ with strictly positive coefficients. Note that
  in $\mathbb{R}^2$ both classes of cones agree, but not in higher
  dimensions. This observation will be useful in the appendix.

  (2) We do not know if
  our results hold for nonlinear perturbations of the affine IFS's
  we study. In studying nonlinear, nonconformal IFS's  one usually
  needs to assume the so-called ``1-bunching'' condition, which
  guarantees that certain kind of bounded distortion holds, and
  therefore allows control of the shape of the cylinder sets; see
  for example \cite{Falconer1994}. For a linear map $A$, 1-bunching
  is equivalent to $\alpha_2(A) > \bigl(\alpha_1(A)\bigr)^2$. This is exactly
  Hypothesis 2 in \cite{HueterLalley1995} and, as remarked in \S
  \ref{sec:kakeya}, it cannot hold in our setting. More
  specifically, $1$-bunching appears to be necessary to extend Lemma
  \ref{thm:angle} to nonlinear maps.

  (3) Computing the singularity dimension of an arbitrary affine
  IFS is a very difficult problem. Recently Falconer and Miao
  \cite{FalconerMiao2007b} succeeded in finding a closed formula in
  the case all the matrices are upper triangular but, as they
  indicate, in general it is very hard to even obtain good numerical
  estimates. In our setting, one could use Lemma
  \ref{thm:norm_alpha} to obtain rigorous upper and lower bounds,
  but the convergence is extremely slow.

  (4) It would be of interest to find more general conditions for the
  validity of \eqref{K-projection}. In particular, is it true that,
  when $\kappa=2$, \eqref{K-projection} holds whenever the
  singularity dimension is strictly larger than $1$?

  (5) Falconer's Theorem shows that the equality of Hausdorff dimension and
  singular value dimension of a self-affine set is typical from the
  point of view of measure, at least when the norms of the linear
  maps do not exceed $\frac{1}{2}$, but does not say anything about
  the topological structure of the exceptional set. In every known
  counterexample, the linear parts of the affine maps commute; this
  is of course a nowhere dense condition. Our results provide some
  support to the conjecture that Minkowski dimension and singular value
  dimension agree for an open and dense family of affine IFS's.
\end{remark}

\begin{ack}
  PS wishes to thank Nuno Luzia and Boris Solomyak for helpful conversations and
  comments.
\end{ack}

\appendix

\section{Tractable self-affine sets} \label{sec:tractable}

It was recently proved in \cite{KaenmakiVilppolainen2006} that the
positivity of the Hausdorff measure is equivalent to a specific
separation condition in a setting going beyond the conformal case.
Working on $\R^d$, we define a nontrivial class of affine IFS's having
this property.
With the notation $\HH^t$, we mean the $t$-dimensional Hausdorff
measure, see \cite[\S 4]{Mattila1995}, and $\QQ_d$ is the family of all
vectors $v \in \R^d$ with strictly positive coefficients..

\begin{definition}
  If for each $i \in I$ there are a contractive invertible matrix $A_i
  \in \R^{d \times d}$ with $||A_i|| \le \ualpha<1$ and a translation
  vector $a_i \in \R^d$ then the collection of affine mappings $\{ A_i
  + a_i : i \in I \}$ is called a \emph{tractable affine iterated
    function system} and the invariant set $E \subset \R^d$ of this
  affine IFS a \emph{tractable self-affine set}
  provided that the condition \eqref{K-coneinvariance}
  is satisfied for the cone $\QQ_d \cup -\QQ_d$ and
      the set $E$ is not contained in any hyperplane of $\R^d$.
\end{definition}

We remark that we do not assume the separation condition
\eqref{K-separation}. To motivate the use of the cone $\QQ_d \cup
-\QQ_d$, recall the explanation in Remark \ref{remarks}(1). The
condition on hyperplanes is simply a non-degeneracy assumption.

We shall show that on a tractable self-affine set the diameter of
a cylinder is comparable to the corresponding largest singular
value. For the proof, we need the following linear algebraic
lemma.

\begin{lemma} \label{thm:enterscone}
  Suppose there is a matrix $A \in \R^{d \times d}$ such that
  \[
  A\bigl(\overline{\QQ_d \cup -\QQ_d}\bigr) \subset \QQ_d \cup -\QQ_d.
  \]
  Then there exist a hyperplane $H$ such that for each
  $w \in \R^d \setminus H$ there is $n_0 \in \N$ with $A^n w \in
  \QQ_d \cup -\QQ_d$ whenever $n \ge n_0$.
\end{lemma}

\begin{proof}
  Let $\{ \lambda_1,\ldots, \lambda_d\}$ be the spectrum of $A$,
  where $|\lambda_1|\ge \ldots\ge |\lambda_n|$. By the
  Perron-Frobenius Theorem, $\lambda_1$ is real and positive, and
  $\lambda_1 > |\lambda_2|$. Moreover, if $v$ is the Perron
  eigenvector associated to $\lambda_1$ then $v \in
  \QQ_d \cup -\QQ_d$. Let $H$ be the hyperplane spanned by all the other
  eigenvectors of $A$  (this is a well-defined hyperplane
  since $\lambda_1$ is a simple, real eigenvalue). Note that $H$ is
  invariant under $A$.

  Fix $w \in \R^d \setminus H$ and write $w = w_1 v
  + h$, where $w_1\neq 0$ and $h \in H$. We have
  \begin{equation} \label{eq:lineariterate}
    A^n w = w_1 \lambda_1^n v + (A|_H)^n(h) = \lambda_1^n \left(w_1 v
      + \lambda_1^{-n} (A|_H)^n(h)\right)
  \end{equation}
  for every $n \in \N$.
  Choose $|\lambda_2| < \mu < \lambda_1$ and $\delta$
  small enough so that $B(w_1 v,\delta) \subset
  \QQ_d \cup -\QQ_d$. Since the spectral radius of $A|_H$ is
  $|\lambda_2|$, we find $n_0 \in \N$ such that $|(A|_H)^n(h)| <
  \mu^n|h|$ and $(\mu/\lambda_1)^n |h| < \delta$ whenever $n \ge n_0$.
  Recalling (\ref{eq:lineariterate}), we conclude that $A^n w \in
  \QQ_d \cup -\QQ_d$ for $n \ge n_0$. The proof is finished.
\end{proof}

\begin{lemma} \label{thm:diamprojection}
  Suppose the collection of affine mappings $\{ A_i + a_i : i \in I
  \}$ is a tractable affine IFS. Then there exists a constant $C \ge
  1$ such that
  \begin{equation*}
    C^{-1}\alpha_1(A_\iii) \le \diam\bigl( E_\iii \bigr) \le
    C\alpha_1(A_\iii)
  \end{equation*}
  whenever $\iii \in I^*$.
\end{lemma}

\begin{proof}
  The diameter of $E_\iii$ is at most a constant times
  $\alpha_1(A_\iii)$ in general, so we only need to prove the other
  direction. Fix $i \in I$ and
  let $H$ be the hyperplane given by Lemma \ref{thm:enterscone}
  applied to the matrix $A_i$. By the tractability,
  the self-affine set $E$ is not contained
  in any translate of $H$. Therefore, the arithmetic difference
  $E-E$ is not contained in $H$ and we can find
  two different points $x,y\in E$ such that $y-x\notin H$.
  Applying Lemma \ref{thm:enterscone}, we find $n$ such that
  $y'-x'\in X(\theta,\beta)$, where
  \[
  y' = (A_i+a_i)^n y \in E, \quad x' = (A_i+a_i)^n x \in E.
  \]
  By Remark \ref{remarks}(1) and Lemma \ref{thm:norm_alpha}, there
  exists a constant $\delta > 0$ such that
  \[
  \diam\bigl(E_\iii\bigr) \ge |A_\iii y'- A_\iii x'| \ge \delta
  |y'-x'| \alpha_1(A_\iii).
  \]
  The proof is complete.
\end{proof}

We introduce in the
following definition a natural separation condition to be used on
tractable self-affine sets. Given a tractable affine IFS, define for
$r>0$
\begin{equation*}
  Z(r) = \bigl\{ \iii \in I^* : \diam\bigl( E_\iii \bigr) \le r <
  \diam\bigl( E_{\iii^-}) \bigr) \bigr\}
\end{equation*}
and if in addition $x \in E$, set
\begin{equation*}
  Z(x,r) = \{ \iii \in Z(r) : E_\iii \cap B(x,r) \ne \emptyset \}.
\end{equation*}

\begin{definition} \label{def:ball_cond}
  We say that a tractable self-affine set $E$ satisfies a \emph{ball
    condition} if there exists a constant $0<\delta<1$ such that for
  each $x \in E$ there is $r_0>0$ such that for every $0<r<r_0$
  there exists a set $\{ x_\iii \in \conv\bigl(E_\iii\bigr) :
  \iii \in Z(x,r) \}$ such that the collection  $\{ B(x_\iii,\delta r)
  : \iii \in Z(x,r) \}$ is disjoint.
  If $r_0>0$ above can be chosen to be infinity for every $x \in E$
  then the tractable self-affine set $E$ is said to satisfy a
  \emph{uniform ball condition}. Here with the notation $\conv(A)$, we
  mean the convex hull of a given set $A$.
\end{definition}

Now we are ready to prove our result concerning tractable
self-affine sets.

\begin{theorem} \label{thm:schief}
  Suppose $E$ is a tractable self-affine set and $P(t)=0$ for some
  $0 < t \le 1$. Then $E$ satisfies the (uniform) ball condition
  if and only if $\HH^t(E)>0$.
\end{theorem}

\begin{proof}
  Notice first that if $0 < t \le 1$ then by Lemma
  \ref{thm:diamprojection}, the topological pressure defined in
  \eqref{eq:pressure} is the same as the topological pressure defined
  in \cite[(3.1)]{KaenmakiVilppolainen2006}.

  Observe that $0 < \diam\bigl( E_\iii \bigr) \le
  \ualpha^{|\iii|} \diam(E)$ for each $\iii \in I^*$.
  According to Remark \ref{remarks}(1) and Lemma \ref{thm:norm_alpha},
  there is a constant $\delta > 0$ for which
  $\alpha_1(A_{\iii\jjj}) \ge \delta \alpha_1(A_\iii)
  \alpha_1(A_\jjj)$ whenever $\iii,\jjj \in I^*$. Hence, using Lemma
  \ref{thm:diamprojection} again, we find a constant $D \ge 1$ such
  that
  \begin{equation*}
    D^{-1} \le \frac{\diam\bigl( E_{\iii\jjj}\bigr)}{\diam\bigl(
    E_\iii \bigr) \diam\bigl(E_\jjj\bigr)} \le D
  \end{equation*}
  for every $\iii,\jjj \in I^*$.
  Since $E_{\iii i} \subset E_\iii$ as $\iii \in
  I^*$ and $i \in I$, we have shown that the collection of compact
  sets $\{ E_\iii : \iii \in I^* \}$ satisfies the assumptions
  (M1)--(M3) introduced in \cite[\S 3]{KaenmakiVilppolainen2006}.

  Using Lemma \ref{thm:diamprojection} once again, we see that there
  exists a constant $C \ge 1$ such that $|A_\iii x - A_\iii y| \le
  C\diam\bigl( E_\iii \bigr) |x-y|$ for every $x,y \in E$ and
  for each $\iii \in I^*$. Therefore, by \cite[Lemma
  5.1]{KaenmakiVilppolainen2006} and \cite[Corollary
  3.10]{KaenmakiVilppolainen2006}, the proof is finished.
\end{proof}





\end{document}